\documentclass{amsart}

\usepackage[english]{babel}

\usepackage{mathrsfs, mathtools, amssymb}

\usepackage{dsfont}

\usepackage{tikz, tikz-cd}
\usepackage[paper=a4paper, margin=2cm]{geometry}

\usepackage{hyperref}
\hypersetup{
    unicode=false,          % non-Latin characters in bookmarks
    pdftoolbar=true,        % show toolbar?
    pdfmenubar=true,        % show menu?
    pdffitwindow=false,     % window fit to page when opened
    pdfstartview={FitH},    % fits the width of the page to the window
    pdftitle={Cohomology Rings of Toric bundles and the ring of conditions},    % title
    pdfauthor={; Khovanskii, A. G.; Monin, L.},     % author
    pdfkeywords={Toric bundles} {Spherical Varieties} {Newton Polyhedra}  {Newton Okounkov bodies}, % list of keywords
    pdfnewwindow=true,      % links in new window
    colorlinks=true,       % false: boxed links; true: colored links
    linkcolor=blue,          % color of internal links
    citecolor=red,        % color of links to bibliography
    filecolor=magenta,      % color of file links
    urlcolor=cyan           % color of external links
}

\theoremstyle{plain}
\newtheorem{theorem}{Theorem}[section]
\newtheorem{lemma}[theorem]{Lemma}
\newtheorem{proposition}[theorem]{Proposition}
\newtheorem{corollary}[theorem]{Corollary}

\theoremstyle{definition}
\newtheorem{definition}[theorem]{Definition}
\newtheorem{remark}[theorem]{Remark}

\newcommand{\rleft}{\mathopen{}\mathclose\bgroup\left}
\newcommand{\rright}{\aftergroup\egroup\right}

\newcommand{\C}{{\mathbb{C}}}
\newcommand{\K}{{\mathbb{K}}}
\newcommand{\R}{{\mathbb{R}}}
\newcommand{\Z}{{\mathbb{Z}}}

\newcommand{\Pm}{{\mathcal{P}}}
\newcommand{\Am}{{\mathcal{AP}}}

\newcommand{\Hm}{{\mathcal{H}}}

\newcommand{\Ann}{{\mathrm{Ann}}}

\newcommand{\diff}{\mathop{}\!d}
\newcommand{\Diff}{{\mathrm{Diff}}}

\newcommand{\Hom}{{\mathrm{Hom}}}

\newcommand{\Vol}{{\mathrm{Vol}}}
\newcommand{\s}{{\mathrm{sign}}}
\newcommand{\tail}{{\mathrm{tail}}}
\newcommand{\ind}{{\mathrm{ind}}}
\newcommand{\rank}{{\mathrm{rank}}}

\begin{document}
\selectlanguage{english}

%%%%%%%%%%%%%%%%%%%%%%%%%%%%%%%%%%%%%%%%%%%%%%%%%%%%%%%%%%%%%%%% 

\title[Generalized virtual polytopes and quasitoric manifolds]{Generalized virtual polytopes and quasitoric manifolds}

\author{Askold Khovanskii}
\address[A.\,Khovanskii]{Department of Mathematics, University of Toronto, Toronto, Canada; Moscow Independent University, Moscow, Russia.}
\email{askold@math.utoronto.ca}
\author{Ivan Limonchenko}
\address[I.\,Limonchenko]{National Research University Higher School of Economics, Russian Federation}
\email{ilimonchenko@hse.ru}
\author{Leonid Monin}
\address[L.\,Monin]{Max Planck Institute for Mathematics in the Sciences, Leipzig, Germany}
\email{leonid.monin@mis.mpg.de}

\subjclass[2020]{57S12, 13F55, 55N45}
\keywords{Quasitoric manifold, starshaped sphere, virtual polytope, multi-fan, multi-polytope, moment-angle-complex, Stanley-Reisner ring}

\begin{abstract}
In this paper we develop a theory of volume polynomials of generalized virtual polytopes based on the study of topology of affine subspace arrangements in a real Euclidean space. We apply this theory to obtain a topological version of the BKK Theorem, the Stanley-Reisner and Pukhlikov-Khovanskii type descriptions for cohomology rings of generalized quasitoric manifolds.  
\end{abstract}

\maketitle

%%%%%%%%%%%%%%%%%%%%%%%%%%%%%%%%%%%%%%%%%%%%%%%%%%%%%%%%%%%%%%%%%

\section{Introduction}
\label{sec:intro}

In~\cite{PK92} Pukhlikov and the first author generalized the classical theory of finitely-additive measures of convex polytopes and proposed a geometric construction for a virtual polytope as a Minkowski difference of two convex polytopes. Based on this notion, in~\cite{KP} the same authors proved a Riemann-Roch type theorem linking integrals and integer sums of quasipolynomials over convex chains from a certain family. As a byproduct, they obtained a description for a cohomology ring of a complex nonsingular projective toric variety via a volume polynomial of a virtual polytope. A theory of mixed volumes of virtual convex bodies was developed in~\cite{timorin1999} in order to produce an 'elementary' proof of the classical g-theorem, motivated by the ideas of~\cite{KP} and the approach of~\cite{McMullen1993OnSP}. 

A topogical generalization of a complex nonsingular projective toric variety is known in toric topology as a (quasi)toric manifold. It was introduced and studied alongside with its real counterpart, a small cover, in~\cite{davis1991convex}: in particular, it was shown that the Stanley-Reisner description for cohomology rings holds for quasitoric manifolds. Since that time quasitoric manifolds and their generalization, torus manifolds~\cite{Masuda1999UnitaryTM,multifan}, have been studied intensively in toric topology and found numerous valuable applications in homotopy theory~\cite{Choi2008QuasitoricMO,hasui2016p,hasui2017p}, unitary~\cite{Buchstaber2007SpacesOP,Lu2016EXAMPLESOQ} and special unitary bordism~\cite{lu2016toric,Limonchenko2017CalabiYH}, hyperbolic geometry~\cite{Buchstaber2016OnMD,Buchstaber2017CohomologicalRO,Baralic2020ToricOA}, and other areas of research. 

A remarkable property of torus manifolds is that they acquire a combinatorial description in a similar way to that in the case of toric varieties. Namely, instead of a fan, it is based on the notions of a multi-fan and a multi-polytope, introduced and studied in~\cite{multifan}. A multi-fan is a collection of cones, which can overlap each other, unlike it was in the classical case for cones in an ordinary fan. A multi-polytope is a multi-fan alongside with a collection of affine hyperplanes orthogonal to the linear spans of its rays. The relation between a multi-polytope and its multi-fan is similar to the one between a polytope and its normal fan. In~\cite{ayzenberg2016volume} the theory of multi-polytopes was applied to prove a version of the BKK Theorem and the Pukhlikov-Khovanskii type description for cohomology rings of quasitoric manifolds. On the other hand, a Stanley-Reisner type description for the cohomology of certain torus manifolds was obtained in~\cite{panovmasuda} using methods and tools of the theory of manifolds with corners and equivariant topology.   

Smooth structures on quasitoric manifolds were constructed in~\cite{Buchstaber2007SpacesOP} by means of a topological analogue of the Cox construction, in which a coordinate subspace arrangement is replaced by a moment-angle manifold. By the result of~\cite{PanUs}, moment-angle-complexes over starshaped spheres have smooth structures. This allowed us in~\cite{KLM21} to introduce the class of generalized quasitoric manifolds consisting of quotient spaces of moment-angle-complexes over starshaped spheres by freely acting compact tori of maximal possible rank. 

This paper is devoted to developing the theory of generalized virtual polytopes and applying it in order to obtain a topological version of the BKK Theorem, the Stanley-Reisner and Pukhlikov-Khovanskii type descriptions for intersection rings of generalized quasitoric manifolds. 

%%%%%%%%%%%%%%%%%%%%%%%%%%%%%%%%%%%%%%%%%%%%%%%%%%%%%%%%%%%%%%%%%%%%%%%%%%%%%%%%%%%%%%%%%%%%%%%%%%%%%%%%%%%

{\bf{Generalized virtual polytopes and affine subspace arrangements.}} The first part of the paper is devoted to the theory of generalized virtual polytopes and integration of forms over them, based on studying the homotopy types of unions of affine subspace arrangements in real Euclidean spaces. The construction and the theory of generalized virtual polytopes were motivated by the properties of integral functionals on the space of smooth convex bodies. We discuss smooth convex bodies in Section 2.

Let $Q$ be a polynomial of degree $\leq k$  (homogeneous polynomial of degree $k$) on $\R^n$, $\omega=d x_1\wedge \dots\wedge d x_n$ be the standard volume form on $\R^n$, and let $C_s$ be the cone of strictly convex bodies $\Delta\subset \R^n$ with smooth boundary. Then the function 
$$
 F(\Delta)=\int_\Delta Q\omega   
$$
on the cone $C_s$ is a polynomial of degree $\leq k+n$  (homogeneous polynomial of degree $k+n$).

Now, to extend the domain of the integration functional to the entire vector space generated by the cone $C_s$, we introduce the notion of a \emph{virtual convex body} as a formal difference of convex bodies (with the usual identification $\Delta_1 - \Delta_2=\Delta_3 - \Delta_4 \Leftrightarrow \Delta_1 + \Delta_4=\Delta_2 + \Delta_3$). Then the following statement summarizes the results of Section 2:

let $M$ be the space of virtual convex bodies representable as a difference of convex bodies from the cone $C_s$. Then the functional $F$ on $C_s$ can be extended as an integral of the form $Q\omega$ over the chain of virtual convex bodies. Moreover, such an extension will be a polynomial on $M$.

In Section 3 we study the homological properties of unions $X$ of (finite) arrangements of affine subspaces $\{L_i\}$ in a real Euclidean space $L=\R^n$ by means of the nerves $K_X$ of their (closed) coverings by $L_i$'s.

Given two affine subspace arrangements indexed by the same finite set of indices $I$, we say that the nerve $K_X$ of the collection $\{L_i\}$ \emph{dominates} the nerve $K_Y$ of the collection $\{M_i\}$ if 
\[
\bigcap_{j\in J}L_j \ne \varnothing \text{ implies that }\bigcap_{j\in J}M_j \ne \varnothing \text{ for any } J\subset I,
\]
and we write $K_X\geq K_Y$ in this case. Furthermore, we say that a continuous map $f\colon X\to Y$ is {\emph{compatible}} with $K_X$ and $K_Y$ if 
\[
x\in L_{i_1}\cap \ldots \cap L_{i_k}\quad \text{ then } \quad f(x)\in M_{i_1}\cap \ldots \cap M_{i_k}.
\]

Our main tool in the study of the homological properties of unions of affine subspaces is the following result:
\begin{itemize}
    \item[(i)] If a map $f\colon X\to Y$ compatible with $K_X$ and $K_Y $ exists, then the condition $K_X\geq K_Y$ holds;
    \item[(ii)] if a map $f\colon X\to Y$ compatible with $K_X$ and $K_Y$ exists, then it is unique up to a homotopy;
    \item[(iii)] if a nerve $K_X$ is isomorphic to a nerve $K_Y$ and a map $f\colon X\to Y$ compatible with $K_X$ and $K_Y$ exists, then $f$ is a homotopy equivalence between $X$ and $Y$. 
\end{itemize}

We then prove that any union $X$ of affine subspaces has the so called {\emph{good triangulation}} (see Definition~\ref{goodtriangdef}) and use this fact to show that if $K_X\geq  K_Y$, then there is a map $f\colon X\to Y$ compatible with $K_X$ and $K_Y$.

Now, suppose we have an arrangement of affine hyperplanes $\{H_i\}$ in $L=R^n$. We call it {\emph{non-degenerate}} if there is no proper linear subspace $V\subset \R^n$ which is parallel to all the hyperplanes $H_i$. Then the union $X$ of such an arrangement has the homotopy type of a wedge of $(n-1)$-dimensional spheres, in which the number of spheres is equal to the number of bounded regions in $L\setminus X$, see also Theorem~\ref{thm:wedgespheres}. Therefore, each cycle $\Gamma\in H_{n-1}(X,\Z)$ can be represented as a linear combination $\Gamma=\sum \lambda_j \partial \Delta_j$, where each coefficient $\lambda_j$ equals the winding number of the cycle $\Gamma$ around a point $a_j \in \Delta_j\setminus \partial \Delta_j$. Here, $\Delta_j$ denotes the closure of a bounded open polyhedron, which is a bounded component of $L\setminus X$.

In Section 4 we study the homotopy properties of unions $X$ of (finite) arrangements of affine subspaces $\{L_i\}$ in a real Euclidean space $L=\R^n$ by means of the methods developed in Section 3 and the theory of smooth convex bodies in the space $L$. 

We say that the two hyperplane arrangements, $\Hm_1$ and $\Hm_2$, are {\emph{combinatorially equivalent}} if the corresponding nerves $K_{\Hm_1}, K_{\Hm_2}$ are isomorphic. Let $\Hm = \{H_1,\ldots, H_s\}, \Hm'=\{H_1',\ldots, H_s'\}$ be two combinatorially equivalent hyperplane arrangements, and let $X = \bigcup H_i$ and $Y = \bigcup H_i'$ be the corresponding unions of hyperplanes. Then there exists a canonical homotopy equivalence $f\colon X\to Y$. Moreover, we show that for any (finite) simplicial complex $K$ there exists a (finite) affine subspace arrangement $\{L_i\}$ such that the nerve of the (closed) covering of $X$ by the $L_i$'s is homotopy equivalent to $X$ and has the homotopy type of the simplicial complex $K$. 

In order to study the homotopy type of a union of affine subspaces in $\R^n$, we consider finite unions $U\subset \R^n$ of open convex bodies: $U=\bigcup U_i$; our goal is reduced to studying the homotopy type of the set $\R^n\setminus U$. We will do it making use of the following notion from convex geometry. By a {\emph{tail cone}} $\tail (U)$ of a convex body $U$ we mean the set of points $v\in \R^n$ such that for any $a\in U$ and $t\geq 0$, the inclusion $a+tv\in U$ holds.

It is easy to see that, for any convex set $U\subset \R^n$, its tail cone $\tail (U)$ has the following properties:

\begin{itemize}
    \item  The set $\tail (U)$ is a convex closed cone in $\R^n$. A convex set $U$ is bounded if and only if $\tail(U)$ is the origin $O\in \R^n$;
    \item If $\tail(U)$ is a vector space $V$, then for any transversal space $V^\prime$ (i.e. for any $V^\prime$ such that $\R^n=V \oplus V^\prime$), the set $U$ can be represented in the form $U = U^\prime \oplus V$, where $U^\prime=U\cap V^\prime$ is a bounded convex set. That is, if $\tail(U)$ is a vector space, then one has: $U=U^\prime \oplus \tail(U)$ for a certain bounded convex set $U^\prime$.
\end{itemize}

Our main result here can be stated as follows:

the set $\R^n\setminus U$ is homotopy equivalent to the set $\R^n\setminus \bigcup \{a_i+\tail(U_i)\}$, where the summation is taken over all $i$ such that $\tail (U_i)$ is a vector space.

Now, assume that all the linear spaces $V_i=\tail(U_i)$ above are equal to the same linear space $V$ and denote by $T$ a subspace transversal to $V$, i.e. such a linear subspace of $\R^n$ that $\R^n=T\oplus V$. Then the set $\R^n\setminus U$ is homotopy equivalent to $T\setminus \{b_i\}$, where $b_i:=T\cap \{a_i+V_i\}$. This statement totally describes the homotopy type of the set $\R^n\setminus \bigcup H_i$, where $\{H_i\}$ is any collection of affine hyperplanes in $\R^n$. Indeed, the complement $\R^n\setminus \bigcup H_i$ is a union of open convex sets. Moreover, the maximal linear subspaces contained in $\tail (U_i)$ are the same for each $U_i$: each of them is equal to the intersection of the linear spaces $\tilde H_i$ parallel to the affine hyperplanes $H_i$.

%%%%%%%%%%%%%%%%%%%%%%%%%%%%%%%%%%%%%%%%%%%%%%%%%%%%%%%%%%%%%%%%%%%%%%%%%%%%%%%%%%%%%%%%%%%%%%%%%%%%%%%%%%%%%%%%%%%%%

{\bf{Volumes of generalized virtual polytopes and intersection rings of generalized quasitoric manifolds.}} In the second part of the paper we apply the theory of volume polynomials of generalized virtual polytopes to study the cohomology rings of generalized quasitoric manifolds.

First, we construct a special cellular structure for generalized quasitoric manifolds and deduce the monomial and linear relations between characteristic submanifolds of codimension 2 in the their intersection rings. Then we prove a topological version of the BKK Theorem, based on the properties of the volume polynomial for a generalized virtual polytope, which yields a convex-theoretic formula for the self-intersection polynomial on the second cohomology of a generalized quasitoric manifold. Finally, we make use of the BKK Theorem as well as the description of a Poincar\'e duality algebra worked out in~\cite{KP, KhoM} to obtain the Pukhlikov-Khovanskii type description of the cohomology ring of a generalized quasitoric manifold.

In Section 5 we introduce the notion of a generalized virtual polytope and study the properties of integral functionals on the space of generalized virtual polytopes. Suppose $\Delta$ is a triangulation of an $(n-1)$-dimensional sphere on the vertex set $V(\Delta)=\{v_1,\ldots,v_m\}$. In what follows, we will identify a simplex of $\Delta$ with the set of its vertices viewed as a subset in $\{1,2,\ldots,m\}$. 

A map $\lambda\colon V(\Delta) \to (\R^n)^*$ is called a {\emph{characteristic map}} if for any vertices $v_{i_1},\ldots, v_{i_r}$ belonging to the same simplex of $\Delta$ the images $\lambda(v_{i_1}),\ldots, \lambda(v_{i_r})$ are linearly independent (over $\R$). Similarly, one can define the notion of an integer {\emph{characteristic map}} $\lambda\colon V(\Delta) \to (\Z^n)^*$.

Such a map defines an $m$-dimensional family of hyperplane arrangements $\Am$ in the following way. For any $h=(h_1,\ldots,h_m)\in \R^m$, the arrangement $\Am(h)$ is given by
\[
\Am(h) = \{H_1, \ldots, H_m\} \text{ with } H_i = \{\ell_i(x) = h_i\},
\]
where we denote by $\ell_i$ the linear function $\lambda(v_i)$ for each $i\in [m]:=\{1,2,\ldots,m\}$. Given a subset $I\subset [m]$, we also denote $H_I = \bigcap_{j\in I} H_j$. If $I\in\Delta$, then $\Gamma_I$ denotes the face dual to $I$ in the polyhedral complex $\Delta^\perp$ dual to the simplicial complex $\Delta$. By definition, facets of $\Delta^\perp$ are closed stars in $\Delta^\prime$ of the vertices of $\Delta$ viewed as vertices of its barycentric subdivision $\Delta^\prime$.

By a {\emph{generalized virtual polytope}} we mean a map $f\colon \Delta^\perp\to  \bigcup_{\Am(h)} H_i$ {\emph{subordinate}} to the characteristic map $\lambda$; that is, for any $I\subset [m]$, we have: 
$$
f(\Gamma_I)\subset H_I.
$$

Let $U$ be a bounded region of $\R^n\setminus \bigcup_{\Am(h)} H_i$ and $W(U,f)$ be a winding number of a map $f$. Given a polynomial $Q$ on $\R^n$, let us consider the following integral functional on the space of generalized virtual polytopes:
\[
I_Q(f):=\sum W(U,f)\int_U Q\omega.
\]

The key result of Section 5 is the computation of all partial derivatives of $I_Q(f)$, leading us to the following statement. Let $I = \{ i_1, \ldots, i_r\} \subseteq [m]$ be such that $I\notin\Delta$ and $k_1, \ldots, k_r$ be positive integers. Then we have
  \[
    \partial_{i_1}^{k_1} \cdots \partial_{i_r}^{k_r} \left(I_Q\right)(f) = 0 \text{.}
  \]
However, if $r = n = \dim\Delta+1$ and $I$ is a simplex in $\Delta$ dual to the vertex $A\in \Delta^{\perp}$, then we have 
  \[
    \partial_I \left(I_Q\right)(f) = \s(I)  Q(A) \cdot |\det (e_{i_1}, \ldots, e_{i_n})| \text{.}
  \] 

We observe that the volume of the oriented image $f_{h}(\Delta^{\perp})\subset\R^n$ is a function on the real vector space $\mathcal L=\{f_{h}\colon \Delta^{\perp}\to\R^n\}$ and its value $Vol({f_h})$ on a generalized virtual polytope $f_{h}$ is a homogeneous polynomial in $h_{1},\ldots,h_{m}$ of degree $n$. This observation and the previous result yield the values of all the partial derivatives of order $n$ for the volume polynomial $\Vol(f_h)$ of the generalized virtual polytope $f_h$ and hence give us this homogeneous polynomial itself. 

We start Section 6 by recalling the notion of a generalized quasitoric manifold introduced in~\cite{KLM21}. In what follows we assume that $K=K_\Sigma$ is a {\emph{starshaped sphere}}, i.e. an intersection of a complete simplicial fan $\Sigma$ in $\R^n\simeq N\otimes_{\Z}\R$ with the unit sphere $S^{n-1}\subset\R^n$. In this case, the moment-angle-complex $\mathcal Z_K$ acquires a smooth structure, see~\cite{PanUs}. Let further, $\Lambda\colon \Sigma(1) \to N$ be a characteristic map. Then the $(m-n)$-dimensional subtorus $H_\Lambda:=\ker\exp\Lambda\subset (S^1)^m$ acts freely on $\mathcal Z_K$ and the smooth manifold $X_{\Sigma,\Lambda}:=\mathcal Z_K/H_\Lambda$ is called a \emph{generalized quasitoric manifold}. 

Our description for the cohomology of $X_{\Sigma,\Lambda}$ goes in three steps:
\begin{itemize}
    \item[(i)] We provide a special cell decomposition for $X_{\Sigma, \Lambda}$ and show that $H^*(X_{\Sigma, \Lambda})$ is generated by the classes of characteristic submanifolds of codimension~$2$;
    \item[(ii)] We deduce the monomial and linear relations between classes of characteristic submanifolds of codimension $2$ in $H_*(X_{\Sigma, \Lambda})$;
    \item[(iii)] We prove a topological version of the BKK Theorem for $X_{\Sigma, \Lambda}$ and then use it to get the Pukhlikov-Khovanskii type description of the intersection ring $H_*(X_{\Sigma, \Lambda})$.
\end{itemize}

It is worth mentioning that the steps (ii) and (iii) above could be used in the much more general setting of torus manifolds. However, in this general case the algebra obtained by the Pukhlikov-Khovanskii description might be different from the intersection ring (cohomology ring). Indeed, the algebra (Theorem~\ref{PKHquasitoric:thm}) computed via the self-intersection polynomial (Theorem~\ref{thm:BKK}) is the Poincar\'e duality quotient of the subalgebra of the cohomology ring generated by classes of characteristic submanifolds of codimension~$2$.

%%%%%%%%%%%%%%%%%%%%%%%%%%%%%%%%%%%%%%%%%%%%%%%%%%%%%%%%%%%%%%%%%%%%%%%%%%%%%%%%%%%%%%%%%%%%%%%%%%%%%%%%%%%%%%%%%%%%%%%%%%%%

\subsection*{Acknowledgements}
We are grateful to Anton Ayzenberg, Victor Buchstaber, Michael Davis, Megumi Harada, Johannes Hofscheier, and Taras Panov for several fruitful and inspiring discussions. The first author is partially supported by the Canadian Grant No.~156833-17. The second author has been funded within the framework of the HSE University Basic Research Program. The second author is also a Young Russian Mathematics award winner and would like to thank its sponsors and jury.

%%%%%%%%%%%%%%%%%%%%%%%%%%%%%%%%%%%%%%%%%%%%%%%%%%%%%%%%%%%%%%%%%

\section{Smooth convex bodies and the space of maps $f:S^{n-1}\to \R^n$}
In this section we consider a motivational construction of smooth virtual convex bodies. 

Consider a set of smooth maps 
$$
f:S^{n-1}\to \R^n.
$$
Such a set forms a vector space under scaling and pointwise addition of functions:
\[
(f_1+f_2)(x) = f_1(x)+f_2(x),\quad (\lambda f)(x) = \lambda f(x).
\]

For a strictly convex smooth  body $\Delta\subset \R^n$, its boundary $\partial \Delta$ can be identified with the image of the unit sphere under a Gauss map 
$$
f_{\Delta}:S^{n-1}\to \partial \Delta.
$$
In terms of the support function $H_\Delta$ of $\Delta$, the map $f_{\Delta}$ is equal to the restriction of the gradient $\mathrm{grad} H_\Delta$ to the sphere $S^{n-1}$. Thus, we got an inclusion of the space of strictly convex smooth bodies (and their formal differences) into the space of smooth mappings from $S^{n-1}$ to $\R^n$. This inclusion respects the Minkowski addition of convex bodies.

We will be interested in integral functionals on the space of convex bodies. First notice, that one can express the integral $\int_\Delta\omega$ in terms of the corresponding map $f_\Delta$:
\[
\int_\Delta\omega=\int_{S^{n-1}}f^*\alpha,
\]
where $\alpha$ is any form such that $d \alpha=\omega$.

Let $\alpha$ be an $(n-1)$-form on $\R^n$ given by 
\[
\alpha=P_1  \widehat {dx_1} \wedge \dots \wedge dx_n+ \dots + P_n  dx_1 \wedge \dots \wedge\ \widehat{dx_n}.
\]
Here, the symbol $\widehat {dx_i}$ means that the term $dx_i$ is missing. The following theorem is obvious.

\begin{theorem}\label{thm:intmaps} If all coefficients $P_i$ of the form $\alpha$ are polynomials of degree $\leq k$ on $\R^n$, then the function $\int_{S^{n-1}} f^*\alpha$ on the space of smooth mappings $f\colon S^{n-1}\rightarrow  \R^n$ is a polynomial of degree $\leq k+n-1$.

If all coefficients $P_i$ of the form $\alpha$ are homogeneous polynomials of degree $k$, then the function $\int_{S^{n-1}} f^*\alpha$ is a homogeneous polynomial of degree $k+n-1$ on the space of smooth mappings.
\end{theorem}

\subsection{Integral functional on the space of maps and winding numbers}
For an $(n-1)$-form $\alpha$ on $\R^n$ and a smooth map $f\colon S^{n-1}\to \R^n$, one can give a different way to compute the integral $\int_{S^{n-1}}f^*\alpha$. Let $U\subset \R^n$ be a connected component of $\R^n\setminus f(S^{n-1})$.

 \begin{definition}The \emph{winding number} $W(U,f)$ of $U$ with respect to $f$ is the mapping degree of the map  
 \begin{equation}\label{eq:wind}
     \frac {f-a}{||f-a||}\colon S^{n-1}\rightarrow S^{n-1}, 
 \end{equation} where $a$ is any point in $U$.
\end{definition}
The mapping degree is well defined, i.e. is independent of the choice of $a\in U$, since maps \eqref{eq:wind} for different $a\in U$ are homotopic to each other.

\begin{proposition}\label{prop:intwind} For any smooth $(n-1)$-form $\alpha $ on $\R^n$ and for any smooth mapping $f\colon S^{n-1}\rightarrow \R^n$ the following identity holds:
$$
\int_{S^{n-1}}f^*\alpha=\sum W(U,f)\int_U d\alpha
 $$
where the sum is taken over all connected components $U$ of the complement \linebreak $\R^n\setminus f(S^{n-1})$.
\end{proposition}

\begin{proof} 
Follows from the Stokes's formula.
\end{proof}

\begin{theorem}\label{thm:polyint} Let $Q$ be a polynomial of degree $\leq k$  (homogeneous polynomial of degree $k$) on $\R^n$ and let $\omega=d x_1\wedge \dots\wedge d x_n$ be the standard volume form on $\R^n$. Then the function
\[
\sum W(U,f)\int_U Q\omega
\]
on the space of smooth mappings is a polynomial of degree $\leq k+n$  (homogeneous polynomial of degree $k+n$).
\end{theorem}

\begin{proof} Consider an $(n-1)$ form $\alpha=P d x_2\wedge \dots\wedge d x_n$, where $P$ is a degree $k+1$ polynomial such that  $\partial P/\partial x_1=Q$. Clearly, $d\alpha=Q\omega$. Thus the statement follows from Theorem~\ref{thm:intmaps} and Proposition~\ref{prop:intwind}.
\end{proof}

Let us denote by $C_s$ the cone of strictly convex bodies $\Delta\subset \R^n$ with smooth boundaries. As a corollary, we obtain the following result.
\begin{corollary}
Let $Q$ and $\omega$ be the same as before. Then the function 
\begin{equation}\label{eq:convint}
 F(\Delta)=\int_\Delta Q\omega   
\end{equation}
on the cone $C_s$ is a polynomial of degree $\leq k+n$  (homogeneous polynomial of degree $k+n$).
\end{corollary} 

\begin{proof} Indeed, for the map $f=\mathrm{grad}\,H_\Delta\colon S^{n-1}\rightarrow \R^n$ there are exactly two connected components of $\R^n\setminus f(S^{n-1})$: the component $U_1=\R^n\setminus \Delta$ and the component $U_2 = int(\Delta)$. Moreover, the corresponding winding numbers are
$$
W(U_1,f)=0;\ \  W(U_2,f)=1.
$$
Thus, the statement follows from Theorem~\ref{thm:polyint}.
\end{proof}

We would like to extend the integration functional to the vector space generated by the cone $C_s$.  
 \begin{definition} 1) A \emph{virtual convex body} is a formal difference of convex bodies (with the usual identification $\Delta_1- \Delta_2=\Delta_3- \Delta_4 \Leftrightarrow \Delta_1+ \Delta_4=\Delta_2+ \Delta_3$);

2) A {\it support function of a virtual convex body} $\Delta=\Delta_1- \Delta_2$ is the difference of the support functions of $\Delta_1$ and $\Delta_2$;

3) A {\it chain of virtual convex bodies with a smooth support function $H$} is the set of connected components $U$ of the complement $\R^n\setminus\mathrm{grad}\,H(S^{n-1})$ taken with the coefficients $W(U,\mathrm{grad}\,H)$.
\end{definition}

The following theorem summarizes the results of this section.
\begin{theorem} Let $M$ be the space of virtual convex bodies representable as differences between convex bodies from the cone $C_s$. Then the function \eqref{eq:convint} on $C_s$ can be extended to $M$ as an integral of the form $Q\omega$ over the chain of virtual convex bodies. Moreover, such an extension is given by a polynomial on $M$.
\end{theorem}

\section{Unions of affine subspaces}

In this section we study homological properties of unions of (finite) affine subspace arrangements in a vector space $L\simeq \R^n$. Let $I$ be a finite set of indices. Consider a set $\{L_i\}$ of affine subspaces in a vector space $L$ indexed by elements $i\in I$ and let $X=\cup_{i\in I}L_i$ be their union. 

First, we define the main combinatorial invariant of a union of a collection affine subspaces. Note that the topological space $X$ has a natural covering by the affine subspaces~$L_i$. 

\begin{definition} The nerve $K_X$ of the natural covering of $X$ is the simplicial complex with vertex set indexed by $I$, i.e. one vertex for each index $i\in I$. A set of vertices $v_{i_1},\dots,v_{i_k}$ defines a simplex in $K_X$ if and only if the intersection $L_{i_1}\cap \dots\cap L_{i_k}$ is not empty.
\end{definition}

Consider another collection of affine subspaces $\{M_i\}$ in a vector space $M$ indexed by the same set of indices $I$ and with the complex $K_Y$ corresponding to the natural covering of $Y$.

\begin{definition} We will say that the nerve $K_X$ of the collection $\{L_i\}$ \emph{dominates} the nerve $K_Y$ of the collection $\{M_i\}$ if 
\[
\bigcap_{j\in J}L_j \ne \varnothing \text{ implies that }\bigcap_{j\in J}M_j \ne \varnothing \text{ for any } J\subset I.
\]
We will write $K_X\geq K_Y$ in this case.

We say that the nerves $K_X$ and $K_Y$ are \emph{equivalent} if $K_X\geq K_Y$ and $K_Y\geq K_X$.
\end{definition}

Note that if $K_X\geq K_Y$, then there is a natural inclusion $K_X\to K_Y$. Moreover, if $K_X$ and $K_Y$ are equivalent, then this inclusion provides an isomorphism between these complexes.

\subsection{Maps compatible with coverings}
In this subsection we introduce our main tool in the study of unions of affine subspace arrangements. Let as before $X = \cup_{i\in I}L_i$ and $Y = \cup_{i\in I}M_i$ be two collections of affine subspaces indexed by a finite set $I$. First, we will need the next definition.

\begin{definition}\label{def:I(x)} For a point $x\in X=\cup_{i\in I} L_i$, let $I(x)$ be the subset of indices in $I$ such that
\[
x\in L_i \text{ if and only if } i\in I(x).
\]
For two points, $x\in X$ and $y\in Y$, we write $x\geq y$ if $I(x)\supset I(y)$.
\end{definition}

In particular, Definition~\ref{def:I(x)} leads us to the following notion.

\begin{definition} A continuous map $f\colon X\to Y$ is {\emph{compatible}} with $K_X$ and $K_Y$ if for any $x\in X$, we have $x\leq f(x)$. In other words, if 
\[
x\in L_{i_1}\cap \ldots \cap L_{i_k}\quad \text{ then } \quad f(x)\in M_{i_1}\cap \ldots \cap M_{i_k}.
\]
\end{definition}

The following theorem is our main tool in the study of homological properties of unions of affine subspace arrangements.

\begin{theorem}\label{thm1} The following statements hold.
\begin{itemize}
    \item[(i)] If a map $f\colon X\to Y$ compatible with $K_X$ and $K_Y$ exists, then the condition $K_X\geq K_Y$ holds;
    \item[(ii)] If a map $f\colon X\to Y$ compatible with $K_X$ and $K_Y$ exists, then it is unique, up to a homotopy;
    \item[(iii)] If $K_X$ is isomorphic to $K_Y$, then the map $f\colon X\to Y$ compatible with $K_X$ and $K_Y$ provides a homotopy equivalence between $X$ and $Y$.
\end{itemize}
\end{theorem}
\begin{proof} (i) Assume a map $f\colon X\to Y$ compatible with $K_X$ and $K_Y$ exists. Then $K_X\geq K_Y$. Indeed, if $L_{i_1}\cap\dots \cap L_{i_k}$ is not empty and contains a point $x$, then the set $M_{i_1}\cap \dots\cap M_{i_k}$ contains $f(x)$ and, in particular, is non-empty.

(ii) if $f,g$ are two maps from $X$ to $Y$ compatible with $K_X$ and $K_Y$, then for any $0\leq t \leq 1$ the map $t f+(1-t)g$ is also compatible with $K_X$ and $K_Y$. Indeed, for any $x\in X$, the set of points $y\in Y$ such that $I(x)\subset I(y)$ is convex.

(iii) Assume that $K_X$ and $K_Y$ are isomorphic and there are maps $f\colon X\to Y$ and $g\colon Y\to X$ compatible with $K_X$ and $K_Y$.

Then the map $g\circ f\colon X\to X$ is a homotopy equivalence. Indeed, the identity map $Id_X$ and the composition map $g\circ f$ are compatible with $K_X$ and hence are homotopy equivalent by (ii). Similarly, the composition map $f \circ g\colon Y \to Y$ is homotopic to the identity map $Id_Y$.
\end{proof}

To prove the existence of compatible maps we will need the following definitions.

\begin{definition}\label{goodtriangdef} A {\emph{good triangulation}} of the set $X=\cup_{i\in I}L_i$ is a triangulation such that the following condition holds: the set of vertices of a simplex $S$ in a good triangulation is totally ordered in the sense of Definition~\ref{def:I(x)}. In other words, there is an order of the set of vertices $\{v_{i_1},\ldots, v_{i_s}\}$ of $S$ such that
\[
I(v_{i_1}) \subset \ldots \subset I(v_{i_s}).
\]
\end{definition}

\begin{definition} Consider the following {\it natural stratification} of $X=\bigcup_{i\in I} L_i$ by open strata of different dimensions: we say that two points $x,y \in X$ belong to one stratum if 

$$x\geq y\quad  and \quad y\geq x,$$ 

or equivalently, $$I(x)=I(y).$$

The stratum containing the point $x$ is the intersection $L(x)$ of the subspaces $L_i$, for all $i\in I(x)$, with removed union of subspaces $L_i$, for all $i \notin I(x)$.
\end{definition}

\begin{definition} A stratum $U_1$ of the natural stratification of $X$ is {\it bigger} than a stratum $U_2$ of the same stratification ($U_1\geq U_2$) if the closure of $U_1$ contains $U_2$.
\end{definition}

It is easy to see that $U_1\geq U_2$ if and only if for any $x\in U_1$, $y\in U_2$ the relation $x\geq y$ holds.

\begin{definition} A stratum $U$ has rank $k$ if the longest possible chain of strictly decreasing strata
$$U=U_1>\dots>U_k$$ has length $k$.
\end{definition}

\begin{theorem} For any finite union $X=\bigcup L_i$ of affine subspaces $L_i$ in a linear space $L$, one can construct a good triangulation of $X$.
\end{theorem}

\begin{proof} We construct a good triangulation for $X$ in two steps. First, we construct a triangulation compatible with the natural stratification of $X$, i.e. a triangulation such that any open simplex is contained in a certain open stratum. 

A triangulation compatible with the natural stratification of $X$ can be constructed inductively by first triangulating all strata of rank one (i.e. all closed strata) and then extending it to all strata of one higher at each step.

Then one can construct a good triangulation for $X$ by taking a barycentric subdivision of any triangulation of $X$ compatible with the standard (natural) stratification. Indeed, the set of vertices of each simplex in this subdivision corresponds to an increasing chain of faces of a simplex in the original triangulation, which are contained in an increasing chain of strata.
\end{proof}

\begin{theorem} If $K_X\geq  K_Y$, then there exists a map $f\colon X\to Y$ compatible with $K_X$ and $K_Y$.
\end{theorem}
\begin{proof} 
 First, let us consider a good triangulation $\tau$ of $X$. Then, for any vertex $v$ of $\tau$, let us define the value $f(v)$ to be any point in $Y$ such that $I(f(x))\supset I(x)$. Such a point always exists, since $K_X \geq K_Y$. Then we can extend the map $f$ linearly on each simplex of $\tau$.
 
The map $f$ constructed above is compatible with $K_X$ and $K_Y$. Indeed, for any point $x\in X$, there exists the smallest simplex $S$ of the good triangulation for $X$ such that $x \in S$. Among the vertices $V(S)$ of this simplex $S$, there exists the biggest vertex $v$. It is easy to see that $I(x)=I(v)$. Since $f(x)$ belongs to the linear combination of the points $f(v_i)$ with $v_i\in V(S)$, the inclusion $I(f(x))\supset I(x)$ holds.
\end{proof}

%Let $I$ be a finite set of indexes. Consider a set $\{L_i\}$  of affine hyperplanes of a linear space $L$ indexed by elements $i\in I$.  We will discuss the homotopy type of the union $X=\cup_{i\in I}L_i$.

% \begin{definition} The union $X$ of hyperplanes $L_i$ is {\emph{non-degenerate}} if the intersection of the linear spaces $\hat L_i$ parallel to the affine hyperplanes $L_i$ is equal to zero.
% \end{definition}

%  Then and if $\dim L=n$ the $L\setminus X$ is a disjoint union of possibly unbounded open convex polyhedra. If $X$ is non degenerate each unbounded component of its complement does non contain any real line.

% \begin{theorem} If $X$ is a non-degenerate union of affine hyperplanes, then $X$ is homotopy equivalent to a wedge of $(n-1)$-dimensional spheres, where $n=\dim L$.

% Each sphere in the wedge could be naturally identified with the boundary of the convex polyhedron, whose interior is a component of $L\setminus X$.
% \end{theorem}

% \begin{proof}

% \end{proof}
% {\bf SKETCH OF PROOF.} Consider a set $L\setminus A$ where $A$ is a finite set containing exactly one point in any connected bounded component of $L\setminus X$.

% One can prove by hands a homopoty equivalence between $L\setminus X$ and $L\setminus A$.

% One can see that the set $L\setminus A$ is homotopy equivalent to a bucket of spheres with needed  properties

% Let 
% \[
% X=\bigcup L_i\quad  and let \quad \hat L=\bigcap \hat L_i
% \]
% where $\hat L_i$ are the   linear spaces  parallel to the affine hyperplanes $L_i$.
% The following lemma is obvious.

\subsection{Barycentric subdivision and a covering of a simplicial complex}

We will need some general facts related to barycentric subdivisions of simplicial complexes.

Let $C'$ be the simplicial complex obtained by the barycentric subdivision of a given simplicial complex $C$. Each vertex of $C'$ is the barycenter of a certain simplex of $C$. A set of vertices of $C'$ belongs to one simplex of $C'$ if and only if the simplices of $C$ corresponding to these vertices are totally ordered with respect to inclusion.

To each vertex $v$ of $C$ let us associate the closed subset $X_v$ of the complex $C'$ equal to the union of all simplices of $C'$ containing the vertex $v$.

\begin{lemma}\label{lem:4} 1) The nerve of the covering of $C'$ by the collection of closed subsets $X_v$ corresponding to all vertices $v$ of $C$ coincides with the original complex $C$.

2) All sets $X_v$ and their nonempty intersections are homotopy equivalent to a point.
\end{lemma}
\begin{proof} 1) By definition, the set of vertices $v$ of $C$ can by identified with the set of subsets $X_v$, which provides a covering of $C'$. If vertices $v_1, \dots,v_k$ belong to one simplex of $C$, then the sets $X_{v_1},\dots, X_{v_k}$ contain the
 barycenter of that simplex, and therefore, these sets have a nonempty intersection.
 
 Conversely a set $X_v$ intersects a simplex $\Delta$ of the complex $C$ only if $v$ is a vertex of $\Delta$. Thus, if the intersection $X_{v_1}\cap\dots\cap X_{v_k}$ is not empty, then $v_1,\dots, v_k$ belong to a simplex $\Delta$ of $C$.

2) Any nonempty intersection $X_{v_1}\cap\dots \cap X_{v_k}$ can be represented as a union of some simplices of $C'$ containing a common vertex, which is the barycenter of the simplex with the vertices $v_1,\dots,v_k$. Observe that such a union is a cone, hence it is homotopy equivalent to a point.
\end{proof}

\subsection{Maps $f\colon K_X'\to Y$ in the case $K_X\geq K_Y$}

A continuous map $f\colon K_X'\to Y$ is compatible with the natural coverings $$B K_X=\cup_{i \in I} X_{v_i}\quad  and \quad Y=\cup_{i\in I}M_i$$ if, for any $i\in I$, the inclusion $$f(X_{v_i})\subset M_i$$ holds.

Suppose that $K_X\geq K_Y$.
Let $K_X$ be the nerve of the natural covering of $$X=\cup_{i \in I}L_i.$$ The barycentric subdivision $K_X'$ of $K_X$ has its own natural covering by the sets $\hat L_i$ equal to the union of the simplices in $K_X'$, which contain the vertex $v_i$ corresponding to the space $L_i$. By Lemma~\ref{lem:4}, the nerve of this covering of $K_X'$ is isomorphic to $K_X$. Now, let us generalize the definition of a map between topological spaces compatible with their coverings. Suppose $I$ is a finite set of indices. Consider a set $\{X_i\}$ of closed subsets of $X$ indexed by elements $i\in I$.

\begin{definition} The nerve $K_X$ of the covering $X=\bigcup X_i$ is the simplicial complex whose set of vertices $V_X$ contains one vertex $v_i$ for each subset $X_i$ i.e. one vertex for each index $i\in I$. A set of vertices $v_{i_1},\dots,v_{i_k}$ defines a simplex in $K_X$ if and only if the intersection $X_{i_1}\cap \dots\cap X_{i_k}$ is not empty.
\end{definition}

The following Theorem can be proved exactly the same way as Theorem~\ref{thm1}.

\begin{theorem} 1) A map $f\colon K_X'\to Y$ compatible with $K_X$ and $K_Y $ exists if and  only if the condition $K_X\geq K_Y$ holds;

2) If a map compatible with $K_X$ and $K_Y$ exists, then it is unique, up to a homotopy.
\end{theorem}

%%%%%%%%%%%%%%%%%%%%%%%%%%%%%%%%%%%%%%%%%%%%%%%%%%%%%%%%%%%%%%%%%
%%%%%%%%%%%%%%%%%%%%%%%%%%%%%%%%%%%%%%%%%%%%%%%%%%%%%%%%%%%%%%%%%

\section{Homotopy type of a union of an affine subspace arrangement}

In this section we study the homotopy type of a union of affine subspace arrangement. In particular, we show that a union of a collection of affine subspaces can have a homotopy type of any simplicial complex (Theorem~\ref{thm:nerve}), whereas a union of affine hyperplanes is always homotopic to a wedge of spheres (Theorem~\ref{thm:wedgespheres}).

Consider a finite set $\{A_i\}$ of affinely independent points in a real vector space $L$. Let $T\subset L$ be the simplex on the vertex set $\{A_i\}$. Alongside with each face $T_J$ of $T$ consider the affine hull $L_{T_J}$ of $T_J$. We obtain a collection of affine subspaces in $L$ corresponding to the faces $T_J$.

% [Is $T_i$ just the face opposite to the vertex $A_i$?]

% [$\Delta_i=T_i$? $L_{T_i}$ is the affine hull of $\{A_i\}$]

Recall that the subspace $A$ of a topological space $X$ is called a {\emph{strong deformation retract}} of $X$ if there is a homotopy $\pi(x,t)\colon X\times I \to X$ such that
\begin{itemize}
    \item[(i)] $\pi (x,0) = x$ for any $x\in X$;
    \item[(ii)] $\pi(x,1)\in A$ for any $x\in X$;
    \item[(iii)] $\pi(a,t)=a$ for any $a\in A$ and $t\in I$.
\end{itemize}
\begin{lemma}\label{lem:defret} The simplex $T$ is a strong deformation retract of the union of hyperplanes in $L$. Moreover, the deformation retraction $\pi\colon L\times I \to L$ can be chosen to preserve the covering of $L$ by affine subspaces $L_{T_i}$, i.e.
\[
\pi(x,t)\in L_{T_i} \text{ for every } x\in L_{T_i}, t\in I.
\]
\end{lemma}
\begin{proof} Note that each point $x\in L$ is representable in a unique way as
$$
x=\sum \lambda_i A_i, \quad  \text{ where }\quad  \sum \lambda_i=1
$$
(the numbers $\lambda_i$ are the barycentric coordinates of $x$ with respect to the simplex $T$).

Consider the projection $p\colon L\to T$, which maps a point $x$ with the barycentric coordinates $\{\lambda_i\}$ to the point $p(x)$ whose $i$-th barycentric coordinate is equal to $\max (\lambda_i,0)$.

It is easy to see that the map $\pi(x,t)$ defined by 
$$
\pi(x,t)=(1-t)x+tp(x)
$$
satisfies the conditions of the Lemma.
\end{proof}

Let $\{T_i\}$ be an ordered collection of faces of the simplex $T$ of size $N$. Consider the following two sets equipped with the covering by $N$ closed convex sets:
\begin{itemize}
    \item the union $\bigcup_{i=1}^N T_i$, equipped with the covering by the faces $T_i$ from the set $\{T_i\}$;
    \item the union $\bigcup_{i=1}^N L_{T_i}$ of affine hulls $L_{T_i}$ of faces $T_i$, equipped with the covering by the spaces $L_{T_i}$.
\end{itemize}

\begin{theorem}\label{thm:homtype} The natural embedding $\bigcup T_i\to \bigcup L_{T_i}$ makes $\bigcup T_i$ a strong deformation retract of $\bigcup L_{T_i}$. Moreover, the deformation retraction can be chosen to preserve the covering of $\bigcup L_{T_i}$ by the affine spaces $L_{T_i}$
%Moreover, there is a projection $\pi\colon \bigcup L_{T_i}\to \bigcup T_i$ which respects the coverings of these spaces and a homotopy between the identity map and the projection which respects the coverings as well.
\end{theorem}

%[If $\pi$ is the retraction, then the theorem just says $\bigcup T_i$ is a deformation retract of $\bigcup L_{T_i}$. Moreover, if the deformation homotopy preserves $\bigcup T_i$ pointwisely, the latter is a strong deformation retract of $\bigcup L_{T_i}$]

\begin{proof} 
Indeed, as the required projection and its homotopy one can take the restriction of the homotopy from Lemma~\ref{lem:defret} to the space $\bigcup L_{T_i}$.
\end{proof}

%%%%%%%%%%%%%%%%%%%%%%%%%%%%%%%%%%%%%%%%%%%%%%%%%%%%%%%%%%%%%%%%%%%%%%%%%%%%%%%%%%%%%%%%%%%%%%%%%%%%%%%%%%%%%%%

\subsection{Barycentric subdivision and the corresponding affine subspaces} Let $\Delta$ be a simplicial complex and let $\Delta'$ be its barycentric subdivision. In particular, any simplex $\Delta_i$ in $\Delta$ corresponds to a vertex $A_{\Delta_i}$ of $\Delta'$.

Consider a collection of affinely independent points in a vector space $L$ identified with the vertices of $\Delta'$. Then $\Delta'$ is naturally embedded into the simplex $T$ generated by this collection.

For a vertex $A_i$ of $\Delta$ let its {\emph{star}} $St(A_i)$ be the collection of simplices of $\Delta$ having $A_i$ as a vertex. Each star $St(A_i)$ determines a face $T_i$ of $T$ by taking the convex hull of vertices of $T$ which correspond to simplices in $St(A_i)$.

% \begin{definition}  To a vertex $A_i$ of $\Delta$ one associates the collection of vertices of $\Delta'$ each of which corresponds to a simplex of $\Delta$ containing the vertex $A_i$. The face $T_i\subset T$ corresponds to $A_i$ if the set of vertices of $T_i$ is the set which corresponds to $A_i$.
% \end{definition}

Let $X_\Delta$ be the union of all the faces $T_i\subset T$ corresponding to the vertices of $\Delta$. Then the space $X=X_\Delta$ has a natural covering by the faces $T_i$. On the other hand, let $Y$ be the union of affine hulls $L_{T_i}$ of the faces $T_i$ corresponding to the vertices of $\Delta$. Then the space $Y$ has a natural covering by the subspaces $L_{T_i}$.

The following statement is an immediate corollary of Theorem~\ref{thm:homtype}.
\begin{corollary} The subset $X\subset Y$ is a deformation retract of $Y$. Moreover, the deformation retraction respects the coverings of $X$ by $T_i$ and of $Y$ by $L_{T_i}$.
\end{corollary}

\begin{theorem}\label{thm:nerve} The nerve of the covering of $X$ by $T_i$ can be naturally identified with the nerve of the covering of $Y$ by $L_{T_i}$.

Both of these nerves can be naturally identified with the original simplicial complex $\Delta$.
\end{theorem}

We can consider the barycentric subdivision $\Delta'$ of $\Delta$ as a subcomplex of the complex of all faces of the simplex $T$.  
Denote by $Z$ the union of all simplices in $\Delta'$. The space $Z$ is equipped with the following covering: with every vertex $A_i$ of $\Delta$ one can associate the union $Z_i$ of all (closed) simplices, containing the vertex $A_i$. In other words, $Z_i$ is the union of all the faces of $T$ containing the vertex $A_i$ and belonging to the simplicial complex $\Delta'$.

Observe that under the embedding $Z\to X$, the sets $Z_i$ are identified with $T_i\cap Z$.

\begin{theorem} There exists a map $\pi\colon X \to Z$ such that the following conditions hold:

1) $\pi$ maps each simplex $T_i$ to the set $Z_i$;

2) $\pi$ maps each simplex from $Z$ to itself.
\end{theorem}
\begin{proof} The set $X$ is stratified by its covering $X=\cup T_i$ in the following way. 

Each stratum of this stratification is a nonempty intersection of a certain collection of the sets $T_i$ without all nonempty intersections of the bigger collections of sets $T_i$. In particular, this stratification also stratifies the set $Z\subset X$.

The set of all the strata of the above stratification can be naturally identified with the set of all simplices of $\Delta$. Indeed, the intersection $\cap T_{i_j}$ is nonempty if and only if there is a simplex in $\Delta$ with the vertices $A_{i_j}$.

In other words, the set of all the strata is in one-to-one correspondence with the set of vertices of $\Delta'$, i.e. with the set of vertices of~$T$.

The triangulation of $X$ by the faces of $T$ belonging to $X$ is compatible with the above stratification, i.e. each open simplex of this triangulation is contained in a certain stratum.

Consider the barycentric subdivision of the triangulation constructed above. Note that it provides a good triangulation for our stratification, i.e. each simplex from this triangulation is compatible in the following sense: if two strata contain two vertices of a simplex of the triangulation, then one of the strata belongs to the closure of another.

Now we are ready to define a map $\pi$. The map $\pi$ is a map from $X$ to $Z$ which is linear on each simplex of the barycentric subdivision of the natural triangulation of the space $X$, which maps each vertex $A$ of the triangulation to the vertex of $\Delta'$ corresponding to the stratum, containing the vertex $A$.

One can easily check that the map we just constructed satisfies all conditions of the Theorem, which finishes the proof.
\end{proof}

\begin{theorem} The map $\pi\colon X\to Z\subset X$ is homotopic to the identity map.

Denote by $\tilde \pi$ the restriction of $\pi$ to $Z$. Then $\tilde \pi $ maps $Z$ to itself and this map is homotopic to the identity map.
\end{theorem}
\begin{proof} Observe that if $x\in T_i\subseteq X$, then $\pi (x)$ also belongs to $T_i$, as well as the entire segment joining these two points, due to the definition of the map $\pi$. Therefore, one can define a linear homotopy $F(x,t)=(1-t)x+t\pi(x)$ between the identity map and the map $\pi$.

Furthermore, $\tilde \pi$ maps each simplex of $\Delta'$ to itself. Hence one can define a linear homotopy $G(x,t)=(1-t)x+t \tilde \pi(x)$ between the identity map and the map $\pi$.
\end{proof}

%%%%%%%%%%%%%%%%%%%%%%%%%%%%%%%%%%%%%%%%%%%%%%%%%%%%%%%%%%%%%%%%%%%%%%%%%%%%%%%%%%%%%%%%%%%%%%%%%%%%%%%%%%%%%%%%%%%%%%%%%%

\subsection{Homotopy type of a union of a hyperplane arrangement}
Let $\mathcal{H}=\{H_1,\ldots, H_s\}$ be a collection of affine hyperplanes in $L\simeq\R^n$ indexed by the set $[s]=\{1,\ldots, s\}$.

\begin{definition}
The {\emph{nerve}} $K_\mathcal{H}$ of $\Hm$ is the simplicial complex on $s$ vertices $v_1\ldots, v_s$ such that a set of vertices $v_{i_1},\ldots,v_{i_k}$ defines a simplex in $K_\mathcal{H}$ if and only if the intersection $H_{i_1}\cap\cdots\cap H_{i_k}$ is not empty.
\end{definition}

We will say that two hyperplane arrangements, $\Hm_1$ and $\Hm_2$ are {\emph{combinatorially equivalent}} if the corresponding nerves $K_{\Hm_1}, K_{\Hm_2}$ are isomorphic.

 \begin{theorem}
    Suppose $\Hm = \{H_1,\ldots, H_s\}, \Hm'=\{H_1',\ldots, H_s'\}$ are two combinatorially equivalent hyperplane arrangements, and let $X= \bigcup H_i$ and $Y= \bigcup H_i'$  be the corresponding unions of hyperplanes. Then there exists a canonical homotopy equivalence $f\colon X\to Y$.   
 \end{theorem}
\begin{proof}
 As the canonical homotopy equivalence $f\colon X\to Y$ one can take any continuous map such that
    \[
    f(x) \in H_j' \text{ for every } x\in H_j. \qedhere
    \]
\end{proof}

In particular, there is a canonical isomorphism between homology groups of combinatorially equivalent hyperplane arrangements:
    \[
    f_*\colon H_*(X) \to H_*(Y).
    \]
    
We will say that the collection of hyperplanes $\{H_1,\ldots, H_s\}$ is {\emph{non-degenerate}} if it is a non-degenerate collection of affine subspaces. That is, there is no proper linear subspace $L\subset \R^n$ parallel to all the hyperplanes $H_i$.

\begin{theorem}\label{thm:wedgespheres}
Let $\Hm$ be a non-degenerate arrangement of affine hyperplanes in $\R^n$. Then its union $X$ is homotopy equivalent to a wedge of $(n-1)$-dimensional spheres.

The number of spheres is equal to the number of bounded regions in $\R^n\setminus X$.    
\end{theorem}
\begin{proof}
We will prove a more general result, see Theorem~\ref{thm:homgeneral} and Corollary~\ref{cor:homgeneral}.
\end{proof}

\begin{corollary}\label{cor:hom} Let $L\supset X=\bigcup L_i$ be a non-degenerate union of affine hyperplanes $L_i$. Then, if $n>1$, the group $H_{n-1}(X,\Z)$ is a free Abelian group generated by the cycles $\partial \Delta_j$,
where $\Delta_j$ is a closure of the bounded open polyhedron, which is a bounded component of $L\setminus X$. 

All other groups $H_i(X,\Z)$ for $i>0$ are equal to zero and $H_0(X,\Z)\cong \Z$.
\end{corollary}

According to Corollary~\ref{cor:hom}, each cycle $\Gamma\in H_{n-1}(X,\Z)$ can be represented as a linear combination
$$
\Gamma=\sum \lambda_j \partial \Delta_j.
$$
Moreover, each coefficient $\lambda_j$ equals the winding number of the cycle $\Gamma$ around a point $a_j \in \Delta_j\setminus \partial \Delta_j$.

\begin{lemma} Suppose $L^0\subset L$ be such a linear space that
\[
L=L^0+\hat L,
\]
i.e. $L$ is a direct sum of $\hat L$ and $L^0$.
Let $L_i^0=L_i\cap L^0$ and $X^0=X\cap L^0=\cup L^0_i$.

Then $X^0$ is a union of a non-degenerate arrangement of the affine hyperplanes $L^0_i\subset L^0$. Moreover, $X=X^0\times \hat L$, thus $X$ is homotopy equivalent to a wedge of $(n-1-l)$-dimensional spheres, where $l=\dim \hat L$.
\end{lemma}

Now we are ready to prove Theorem~\ref{thm:wedgespheres}. Let $U\subset \R^n$ be a finite union of open convex bodies: $U=\bigcup U_i$. We are going to study the homotopy type of the set $\R^n\setminus U$. First, we need the following definition.

\begin{definition} The {\emph{tail cone}} of a convex body $U$ is a set of points $v\in \R^n$ such that for any $a\in U$ and $t\geq 0$, the inclusion $a+tv\in U$ holds.
\end{definition}

One can check that for any convex set $U\subset \R^n$ the set $\tail (U)$ satisfies the following conditions.
\begin{itemize}
    \item  The set $\tail (U)$ is a convex closed cone in $\R^n$. A convex set $U$ is bounded if and only if $\tail(U)$ is the origin $O\in \R^n$;
    \item  If $\tail(U)$ is a vector space $L$, then for any transversal space $L_1$ (i.e. for any $L_1$ such that $\R^n=L \oplus L_1$) the set $U$ is representable in the form $U = U_1 \oplus L$, where $U_1=U\cap L_1$ is a bounded convex set. That is, if $\tail(U)$ is a vector space, then one has: $U=U_1\oplus \tail(U)$, for a certain bounded convex set $U_1$.
\end{itemize}

If the set $\tail (U_i)$ is a linear space $L_i$, then alongside with $U_i$ we can also consider a shifted space $a_i + L_i\subset U_i$, where $a_i$ is an arbitrary point in $U_i$.

We will prove the following theorem.

\begin{theorem}\label{thm:homgeneral} The set $\R^n\setminus U$ is homotopy equivalent to the set $\R^n\setminus \bigcup \{a_i+L_i\}$, where the union is taken over all indices $i$ such that $\tail (U_i)$ is a vector space.
\end{theorem}

Suppose that in the above theorem all the linear spaces $L_i$ are equal to the same linear space $L$. Denote by $T$ a transversal subspace to $L$, i.e. such a linear subspace in $\R^n$ that $\R^n=T\oplus L$.

\begin{corollary}\label{cor:homgeneral} Under the above assumptions, the set $\R^n\setminus U$ is homotopy equivalent to $T\setminus \{b_i\}$, where $b_i=T\cap \{a_i+L_i\}$.
\end{corollary}

Note that Corollary~\ref{cor:homgeneral} totally describes the homotopy type of the set $\R^n\setminus \bigcup H_i$, where $\{H_i\}$ is any collection of affine hyperplanes in $\R^n$. Indeed, the complement $\R^n\setminus \bigcup H_i$ is a union of open convex sets. Moreover, the maximal linear subspaces contained in $\tail (U_i)$ are the same for each $U_i$: each of them is equal to the intersection of all the linear spaces $\tilde H_i$ parallel to the affine hyperplanes $H_i$.

To prove the theorem we will need some general facts about convex bodies.

\begin{lemma}
Suppose $U\subset \R^n$ is a bounded open convex set, $X$ is the closure of $U$, $\partial X$ is the boundary of $X$ (one has: $X=U\cup \partial X$), and let $a\in U$ be any point in $U$. Then $\partial X$ is a deformation retract of $X\setminus \{a\}$.
\end{lemma}
\begin{proof} Let $\pi\colon X\setminus \{a\}\to \partial X$ be the projection of $X\setminus \{a\}$ to $\partial X$ from the point $a$. The following map provides a deformation retraction:
$$ 
F(x,t)=(1-t)x+t \pi(x),
$$
where $x\in X\setminus \{a\}$ and $0\leq t \leq 1$.
\end{proof}

% This homotopy has the following properties:

% 1) $\pi(x,0)$ is the identity map,

% 2) for every $0\leq t\leq 1$ and $x\in \partial X$, $\pi(x,t)=x$,

% 3) $\pi(x,1)=\pi(x)$.

% Thus $\pi(x,t)$ provides the needed homotopy.

\begin{corollary} Let $U\subset \R^n$ be an open convex set such that $\tail(U)$ is a vector space $L$. Then, by definition, for any $a\in U$ the shifted space $a+L$ belongs to $U$ and the set $X$ is homotopy equivalent to the set $X\setminus L$, where $X$ is the closure of $U$.
\end{corollary}

We will need the following auxiliary lemma. Let us represent $\R^n$ as $\R^{n-1} \oplus \R^1$ and let us use accordingly the notation $(x,y)$ for points in $\R^n$, where $x\in \R^{n-1}$ and $y\in \R^1$.

Let $y=f(x)$ be a continuous function on $\R^{n-1}$. Denote by $X\subset \R^n$ the set of points $(x,y)$, where $y\geq f(x)$. Then $\partial X$ is the graph of the function $f$ (i.e. $(x,y)\in \partial X$ if and only if $y=f(x)$).

\begin{lemma} The natural projection $\pi\colon X\to \partial X$ mapping a point $(x,y)$ to $(x,f(x))$ is homotopic to the identity map.
\end{lemma}
\begin{proof}
One can consider the following homotopy: 
$G (x,y,t)=(1-t)(x,y)+t \pi(x,y)$.
\end{proof}

% This homotopy has the following properties:

% 1) $\pi(x,y,0)$ is the identity map;

% 2) for every $0\leq t\leq 1$ and $(x,y)\in \partial X$ $\pi(x,y,t)=(x,y)$;

% 3)  $\pi (x,y,1)=\pi(x,y)$.

% Thus $\pi(x,y,t)$ provides the needed homotopy.

Now, suppose that the set $\tail(U)\subset \R^n$ is not a vector space, i.e. assume that there is a vector $v\in \tail(U)$ such that the vector $-v$ does not belong to $\tail(U)$.

Let $a\in U$ be an arbitrary point. Since $-v$ is not in $\tail (U)$, there is a positive number $\tau$ such that $a-\tau v\in \partial X$. Let $\tilde L$ be the supporting hyperplane of $X$ at the point $a-\tau v$.

Let us make an affine change of variables in $\R^n$ in such a way that the hyperplane $\tilde L$ becomes the hyperplane $y=1$, the point $a-\tau v$ becomes the point $(0,1)$ and the vector $v$ becomes the standard basis vector in $\R^1$.

After this change of coordinates, $U$ becomes an open convex set in $\R^{n-1}\oplus \R^1$ such that $U$ belongs to the half space $y\geq 1$, and alongside with every point $a\in U$ our convex set $U$ contains the entire ray $a+\tau v$, where $\tau\geq 0$ and $v$ is the vector $(0,1)$.

Consider the diffeomorphism $g$ of the open half space $y>0$ to itself defined by the following formula: $$g(x,y)=(xy,y).$$

\begin{lemma} Under the diffeomorphism $g$, the closure $X$ of $U$ is mapped to the domain $Y$ defined by the following condition $(x,y)\in Y$ if and only if $y\geq f(x)$, where $f$ is a certain continuous function on $\R^n$.
\end{lemma}
\begin{proof} First, let us consider the map $\tilde g\colon\partial X\to \R^{n-1}\oplus \{0\}$ given by the formula: $$ \tilde g(x,y)=(xy,0).$$ Let us show that $\tilde g$ is a homeomorphism between $\partial X$ and $\R^{n-1}$. For each vector $x\in \R^{n-1}\oplus \{0\}$, consider the set of points $\partial X_x\subset \partial X$ defined by the following condition: a point $(x_0,y_0)\in \partial X_x$ if and only if $x_0$ is proportional to $x$.  It is easy to see that the set $\partial X_x$ is homeomorphic to a line.

  We can parametrize it by an oriented distance from the point $(0,1)$ (which belongs to $\partial X_x$, for every $x$) along this curve with arbitrary chosen orientation.

  Now, the $\tilde g$ maps the curve $\partial X_x$ to the line of scalar multiples of $x$. Moreover, this map is monotonic and proper. Hence it provides a homeomorphism between $\partial X_x$ and the line $\tau x, \tau \in \R$. This argument implies that the map $\tilde g\colon\partial X\to \R^{n-1}$ is a homeomorphism.
  
  Observe that the image of $\partial X$ under the diffeomorphism $g\colon (x,y)\mapsto (xy,y)$ is a graph of the function $f$ such that the value $f(x)$ equals the coordinate $y$ of the point $(x,y):=\tilde g^{-1}(x)$. Then the set $X$ is mapped by this diffeomorphism to the domain in $\R^n$ consisting of the points $(x,y)$, where $y\geq f(x)$.
  \end{proof}

\begin{corollary}
Let $U\subset \R^n$ be an open convex set such that the cone $\tail(U)$ is not a vector space. Then the boundary $\partial X$ of the closure $X$ of $U$ is homotopy equivalent to $X$.
\end{corollary}

\section{Generalized virtual polytopes: definition and results}

Let $\Delta$ be a simplicial complex homeomorphic to the $(n-1)$-dimensional sphere. Denote by $V(\Delta)=\{v_1,\ldots,v_m\}$ the set of vertices of $\Delta$. In what follows, we will identify a simplex $S$ of $\Delta$ with the set of vertices $I\subset V(\Delta)$ which belong to $S$.

\begin{definition}
The map $\lambda\colon V(\Delta) \to (\R^n)^*$ is called a {\emph{characteristic map}} if for any vertices $v_{i_1},\ldots, v_{i_r}$ belonging to the same simplex of $\Delta$ the images $\lambda(v_{i_1}),\ldots, \lambda(v_{i_r})$ are independent. In particular, for any maximal simplex $\{v_{i_1},\ldots, v_{i_n}\}$ the images $\lambda(v_{i_1}),\ldots, \lambda(v_{i_n})$ form a basis of~$(\R^n)^*$.

The map $\lambda\colon V(\Delta) \to (\Z^n)^*$ is called an integer {\emph{characteristic map}} if for any maximal simplex $\{v_{i_1},\ldots, v_{i_n}\}$ of $\Delta$ the images $\lambda(v_{i_1}),\ldots, \lambda(v_{i_n})$ form a basis of the lattice $(\Z^n)^*$.
\end{definition}

Let us denote by $\ell_i$ the linear function $\lambda(v_i)$, for any $i\in[r]$. The characteristic map $\lambda$ defines an $m$-dimensional family of hyperplane arrangements $\Am$ in the following way. For any $h=(h_1,\ldots,h_m)\in \R^m$, the arrangement $\Am(h)$ is given as
\[
\Am(h) = \{H_1, \ldots, H_m\} \text{ with } H_i = \{\ell_i(x) = h_i\}.
\]
We denote by $X_h$ the union of all the hyperplanes from $\Am(h)$.

Given a subset $I\in [s]$, we will denote by $H_I$ the intersection 
\[
H_I = \bigcap_{j\in I} H_j.
\]
It follows from the definition of the characteristic map that $H_I$ is non-empty whenever the vertices $v_j$ with $j\in I$ belong to the same simplex.

Let $\Delta^\perp$ be the {\emph{dual polyhedral complex}} to $\Delta$: we define a correspondence between faces of $\Delta^\perp$ and the strata $H_I$ in the following way. A face $\Gamma_I$ of $\Delta^\perp$, dual to a simplex $I$ of $\Delta$, is associated to the stratum $H_I$.

\begin{definition}
We say that a map $f\colon\Delta^\perp \to X_h$ is {\emph{subordinate}} to a characteristic map $\lambda$, if for any face $\Gamma_I$ of $\Delta^\perp$ we have $f(\Gamma_I)\subset H_I$.
\end{definition}

% Let $X$ be any set, then the set of all maps $f\colon X\to \R^d$ has a structure of a convex cone. Indeed for two maps $f_1,f_2\colon X\to \R^d$ we define
% \[
% (\lambda_1f_1+\lambda_2f_2)\colon x\mapsto \lambda_1f_1(x)+\lambda_2f_2(x) \in\R^d,
% \]
% for any $x\in X$ and $\lambda_1, \lambda_2\in \R_{\geq 0}$.

\begin{theorem}\label{thm:submaps}
The space of maps $f\colon\Delta^\perp \to X_h$ subordinate to a characteristic map $\lambda$ is a non-empty convex set. In particular, any two such maps are homotopic.
\end{theorem}

% We will prove theorem in two steps. Let us first consider the maps $f:\Delta\to \R^n$ that map $\Gamma$ to the corresponding stratum $H_\Gamma$

% \begin{lemma}
% Such maps $f$ exist and they are all homotopic.
% \end{lemma}
\begin{proof}
First, let us show the second part of the statement, assuming that a map $f\colon \Delta^\perp\to X_h$ subordinate to the characteristic map $\lambda$ exists. Observe that $H_I$ is a convex set, for any $I\subset [m]$. Therefore, for any two maps $f, f'\colon \Delta^\perp\to X_h$ subordinate to the characteristic map $\lambda$, any member of the linear homotopy between them is also subordinate to this characteristic map:
$$
f_{t}:=(1-t)f+tf',\quad t\in [0,1].
$$
Thus the space of maps $f\colon \Delta^\perp\to X_h$ compatible with the coverings of $\Delta^{\perp}$ and $X_h$ is contractible (assuming it is non-empty).

To show the existence of such maps, we use the following construction. First, let us choose any inner product on $\R^n$. This defines a set of distinguished points $x_I\in H_I$ via taking orthogonal projection of the origin in $\R^n$ to an affine subspace $H_I$. On the other hand, the points of the polyhedral complex $\Delta^\perp$ dual to the simplicial complex $\Delta$, being the vertices of the barycentic subdivision $\Delta'$, are in a bijection with simplices of $\Delta$, hence are labeled by subsets $I\subset [m]$. 

We construct the map $f_h\colon \Delta^\perp\to X_h$ subordinate to the characteristic map $\lambda$ as follows. First, we define the images of the above mentioned points $v_I$ of the complex $\Delta^\perp$ by the formula
\[
f_h(v_I) = x_I,
\]
and then we extend this map by linearity. Note that the map $f_h$ we have just constructed is well-defined, since $(\Delta,\lambda)$ is a characteristic pair (indeed, $H_I$ is nonempty whenever $I$ corresponds to a simplex in $\Delta$) and is compatible with the covering of the complex $\Delta^{\perp}$ by the stars $St(v_i)$ in $\Delta'$ of the vertices $v_i\in\Delta$, by construction.
\end{proof}

The family of maps $f_h\colon\Delta^\perp\to X_h$ satisfies another nice property.
\begin{corollary}\label{cor:linear}
In the situation as before, one has $f_{h+h'}=f_h+f_{h'}$.
\end{corollary}
\begin{proof}
The statement follows from the fact that the distinguished points $x_I$ used in the above construction depend linearly on $h\in \R^n$:
\[
x_{I,h+h'}= x_{I,h}+x_{I,h'}. \qedhere
\]
\end{proof}

With every affine hyperplane arrangement $\Am(h)$ let us associate a chain $\Delta(h)=\sum_i W(U_i,f)U_i$, where $U_i$ are the connected components of the complement $\R^n\setminus \Am(h)$, and $f\colon\Delta^\perp\to \Am(h)$ is any map subordinate to a characteristic map $\Lambda$. Since any two such maps are homotopic, the chain $\Delta(h)$ is well-defined. 

\begin{definition}\label{def:genvir}
We will call the chain $\Delta(h)$ a {\emph{generalized virtual polytope}} associated to the simplicial complex $\Delta$, characteristic map $\Lambda$ and the vector $h\in\R^m$. We denote by $\Pm_{\Delta,\Lambda}\simeq \R^m$ the space of all generalized virtual polytopes associated to the simplicial complex $\Delta$ and characteristic map $\Lambda$.
\end{definition}

\begin{remark}
Classical virtual polytopes are piecewise linear functions defined not necessarily in the complements of unions of affine hyperplane arrangements (convex chains), hence they carry more information, than a chain $\Delta(h)$. However, in this paper we are interested only in the volumes of generalized virtual polytopes and integrals over them, so it is enough for us to work with the chain $\Delta(h)$. We will study other valuations on the space of generalized virtual polytopes in the future works.
\end{remark}

\subsection{Integration over generalized virtual polytopes}
Let $\alpha$ be an $(n-1)$-form on $\R^n$ given by the formula: 
\[
\alpha=P_1  \widehat {dx_1} \wedge \dots \wedge dx_n+ \dots + P_n  dx_1 \wedge \dots \wedge\ \widehat{dx_n}.
\]
Here, the symbol $\widehat {dx_i}$ means that the term $dx_i$ is missing. The following theorem is obvious.

\begin{theorem} If all the coefficients $P_i$ of the form $\alpha$ are homogeneous polynomials of degree $k$ (a polynomial of degree $\leq k$) on $\R^n$, then the function $\int_{\Delta^\perp} f^*\alpha$ is a homogeneous polynomial of degree $k+n-1$ (a polynomial of degree $\leq k+n-1$) on the space of mappings $f\colon\Delta^\perp\rightarrow  \bigcup_{\Am(h)} H_i$ subordinate to the corresponding characteristic map.
\end{theorem}
\begin{proof}
Analogous to the proof of Theorem~\ref{thm:intmaps}, since by Corollary~\ref{cor:linear} the family of maps $f_h$ can by chosen so that $f_{h_1}+f_{h_2}=f_{h_1+h_2}$.
\end{proof}

Let $U$ be a bounded region in $\R^n\setminus \bigcup_{\Am(h)} H_i$ and $W(U,f)$ be the winding number for a map $f$ as before. The following proposition follows from the Stokes' theorem.
\begin{proposition}
 Let $\alpha$ be as before, $d\alpha = Q\omega$, where $Q$ is a polynomial of degree $\leq k$  (homogeneous polynomial of degree $k$) on $\R^n$ and let $\omega=d x_1\wedge \dots\wedge d x_n$ be the standard volume form on $\R^n$. Then the following identity
\[
\sum W(U,f)\int_U Q\omega = \int_{\Delta^\perp} f^*\alpha
\]
holds for the map $f\colon\Delta^\perp\rightarrow  \bigcup_{\Am(h)} H_i$ subordinate to the corresponding characteristic map.

In particular, $\sum W(U,f)\int_U Q\omega$ is a polynomial of degree $\leq k+n-1$ (homogeneous polynomial of degree $k+n-1$) on the space of mappings $f\colon\Delta^\perp\rightarrow  \bigcup_{\Am(h)} H_i$ subordinate to the corresponding characteristic map.
\end{proposition}

Given a polynomial $Q$ on $\R^n$ and a generalized virtual polytope $f\colon\Delta^\perp\to  \bigcup_{\Am(h)} H_i$, let us denote by $I_Q(f)$ the integral
\[
\sum W(U,f)\int_U Q\omega.
\]
The following lemma computes the (mixed) partial derivatives of $I_Q$.

\begin{lemma}
  \label{Ider}
  Let $f\colon\Delta^\perp\to  \bigcup_{\Am(h)} H_i$ be a generalized virtual polytope given by the simplicial complex $\Delta$ on $s$ vertices. Suppose $I = \{ i_1, \ldots, i_r\} \subseteq \{ 1, \ldots, s \}$ is a subset such that the vertices $v_{i_1},\ldots,v_{i_r}$ do not form a simplex in $\Delta$ and let $k_1, \ldots, k_r$ be positive integers. Then we have:
  \[
    \partial_{i_1}^{k_1} \cdots \partial_{i_r}^{k_r} \left(I_Q\right)(f) = 0 \text{.}
  \]
  However, if $r = n$ and the vertices $v_{i_1}, \ldots, v_{i_n}$ generate a simplex in $\Delta$ dual to the vertex $A\in \Delta^{\perp}$, then we have: 
  \[
    \partial_I \left(I_Q\right)(f) = \s(I)  Q(A) \cdot |\det (e_{i_1}, \ldots, e_{i_n})| \text{.}
  \]
\end{lemma}
\begin{proof}
By linearity of derivation, it is enough to compute the partial derivatives for each summand $W(U,f)\int_U Q\omega$ separately. 

In the first case, when the vertices $v_{i_1},\ldots,v_{i_r}$ do not form a simplex in $\Delta$, the intersection of the corresponding hyperplanes $H_{i_1},\ldots,H_{i_r}$ does not correspond to a vertex of $U$, for any bounded region $U$ in $\R^n \setminus \bigcup_{\Am(h)} H_i$ with $W(U,f)\ne 0$. Hence $\partial_{i_1}^{k_1} \cdots \partial_{i_r}^{k_r} \left(I_Q\right)(f) = 0$ by \cite[Lemma 6.1]{hof2020}.

On the other hand, if the vertices $v_{i_1}, \ldots, v_{i_n}$ generate a simplex in $\Delta$, then there exists exactly one region $U_i$ in $\R^n \setminus \bigcup_{\Am(h)} H_i$ having the intersection
\[
A = H_{i_1}\cap\ldots\cap H_{i_n}
\]
as its vertex. Then by \cite[Lemma 6.1]{hof2020} we get
\[
 \partial_I \left(I_Q\right)(f) =\partial_I \int_{U_i}Q\omega =  \s(I)  Q(A) \cdot |\det (e_{i_1}, \ldots, e_{i_n})|.
\]
\end{proof}

As an immediate consequence of Lemma~\ref{Ider} we obtain the following statement.
\begin{corollary}\label{volpolynomialderivativescoro}
 Let $f\colon\Delta^\perp\to  \bigcup_{\Am(h)} H_i$ be a generalized virtual polytope associated to a simplicial complex $\Delta$ on $s$ vertices. Suppose $I = \{ i_1, \ldots, i_r\} \subseteq \{ 1, \ldots, s \}$ is a subset such that the vertices $v_{i_1},\ldots,v_{i_r}$ do not form a simplex in $\Delta$ and let $k_1, \ldots, k_r$ be positive integers. Then we have:
  \[
    \partial_{i_1}^{k_1} \cdots \partial_{i_r}^{k_r} \Vol(f) = 0 \text{.}
  \]
  However, if $r = n$ and the vertices $v_{i_1}, \ldots, v_{i_n}$ generate a simplex in $\Delta$ dual to the vertex $A\in \Delta^{\perp}$, then we have: 
  \[
    \partial_I \Vol(f)(f) = \s(I) \cdot |\det (e_{i_1}, \ldots, e_{i_n})|\text{.}
  \]
\end{corollary}

\section{Cohomology of generalized quasitoric manifolds}

In this section we will describe the cohomology rings of a class of torus manifolds called generalized quasitoric manifolds. Let $T\simeq (S^1)^n$ be a compact torus with character lattice $M$ and $N=M^\vee$. Suppose $K$ is an abstract simplicial complex of dimension $n-1$ on the vertex set $[m]=\{1,2,\ldots,m\}$. Recall that its moment-angle-complex $\mathcal Z_K$ is defined to be the $(m+n)$-dimensional cellular subspace in the unitary polydisc $(D^2)^m\subset\C^m$ given by the formula 
$\bigcup_{I\in K}\prod\limits_{i=1}^{m}{Y_i}$, where $Y_i=D^2$, if $i\in I$ and $Y_i=S^1$, otherwise. 

There is a natural (coordinatewise) action of the compact torus $(S^1)^m$ on $\mathcal Z_K$ and the orbit space $\mathcal Z_K/(S^1)^m$ is homeomorphic to the cone over the barycentric subdivision of $K$. 

In what follows we assume that $K=K_\Sigma$ is a {\emph{starshaped sphere}}, i.e. an intersection of a complete simplicial fan $\Sigma$ in $\R^n\simeq N\otimes_{\Z}\R$ with the unit sphere $S^{n-1}\subset\R^n$). In this case, the moment-angle-complex $\mathcal Z_K$  acquires a smooth structure, see~\cite{PanUs}.

Let further, $\Lambda\colon \Sigma(1) \to N$ be a {\emph{characteristic map}}, i.e. such a map that the collection of vectors
\[
\Lambda(\rho_1),\ldots, \Lambda(\rho_k) 
\] 
can be completed to a basis of the cocharacter lattice $N$, whenever $\rho_1,\ldots,\rho_k$ generate a cone in $\Sigma$. 
Then the $(m-n)$-dimensional subtorus $H_\Lambda:=\ker\exp\Lambda\subset (S^1)^m$ acts freely on $\mathcal Z_K$ and the smooth manifold $X_{\Sigma,\Lambda}:=\mathcal Z_K/H_\Lambda$ will be called a \emph{generalized quasitoric manifold}. 

Our description of cohomology rings of $X_{\Sigma,\Lambda}$ will be given in three steps:
\begin{itemize}
    \item[(i)] First, we give a cellular decomposition of $X_{\Sigma, \Lambda}$ of a special type and show that $H^*(X_{\Sigma, \Lambda})$ is generated by the classes, dual to the classes of characteristic submanifolds of codimension $2$ in $X_{\Sigma,\Lambda}$;
    \item[(ii)] Then we deduce two sets of relations in the intesection ring of $X_{\Sigma,\Lambda}$ between the classes of characteristic submanifolds of codimension $2$ in $X_{\Sigma,\Lambda}$;
    \item[(iii)] Finally, we prove a topological version of the BKK Theorem for $X_{\Sigma, \Lambda}$ and use it to get a Pukhlikov-Khovanskii type description for the integral cohomology ring $H^*(X_{\Sigma, \Lambda})$.
\end{itemize}

\begin{remark}
Note that the steps (ii) and (iii) above could be made successfully in a much more general class of torus manifolds. However, in this more general case the algebra obtained by a Pukhlikov-Khovanskii description might be different from the cohomology ring. Indeed, the algebra computed via the self-intersection polynomial is the Poincar\'e duality quotient of the subalgebra of the cohomology ring generated by classes dual to the classes of characteristic submanifolds of codimension~$2$ (see \cite{ayzenberg2016volume} for details).
\end{remark}

In what follows, we will always assume that our generalized quasitoric manifolds are \emph{omnioriented}; as in the case of a quasitoric manifold, we say that $X_{\Sigma,\Lambda}$ is omnioriented if an orientation is specified for $X_{\Sigma,\Lambda}$ and for each of the $m$ codimension-2 characteristic submanifolds $D_i$. The choice of this extra data is convenient for two reasons. First, it allows us to view the circle fixing $D_i$ as an element in the lattice $N=\Hom(S^1,T^n)\simeq\Z^n$. But even more importantly, the choice of an omniorientation defines the fundamental class $[X_{\Sigma,\Lambda}]$ of $X_{\Sigma,\Lambda}$ as well as cohomology classes $[D_i]$ dual to the characteristic submanifolds. 

We further assume that $\Sigma \subset \R^n$ and $N_\R\simeq\R^n$ are endowed with orientation. This defines a sign for each collection of rays $\rho_{i_1},\ldots,\rho_{i_n}$ forming a maximal cone of $\Sigma$ in the following way. Let $I=\{i_1,\ldots,i_n\}$ be a set of indices ordered so that the collection of rays $\rho_{i_1},\ldots,\rho_{i_n}$ is positively oriented in $\R^n$. Then
\[
\s(I)= \det(\Lambda(\rho_{i_1}),\ldots, \Lambda(\rho_{i_n})) = \pm 1.
\]

Finally, as before, with a characteristic pair $(\Sigma, \Lambda)$ we associate a space of generalized virtual polytopes $\Pm_{\Sigma,\Lambda}\simeq \R^m$. To any generalized virtual polytope $\Delta(h)\in \Pm_{\Sigma,\Lambda}$ we associate an element of $H^2(X_{\Sigma,\Lambda})$ as follows:
\[
\Delta(h) \mapsto h_1[D_1]+\ldots +h_m[D_m] \in H^2(X_{\Sigma,\Lambda}),
\]
where $D_1,\ldots,D_m$ are the codimension-2 characteristic submanifolds oriented according to the given omniorientation of $X_{\Sigma,\Lambda}$. 

\subsection{Cellular decompositions of generalized quasitoric manifolds.} To provide a cellular decomposition of the generalized quasitoric manifold $X_{\Sigma, \Lambda}$, let us first give a slightly different description of the moment-angle complex $\mathcal Z_K$ for a starshaped sphere $K=K_\Sigma$. Observe that the moment-angle complex is given as a disjoint union of strata $\mathcal Z_K = \bigsqcup_{\sigma \in \Sigma} H_\sigma$, where
\[
H_\sigma= \mathcal{Z}_K \cap \left( \bigcap_{\rho_i\in \sigma} \{z_i=0\}\right) \cap \left(\bigcap_{\rho_j\notin \sigma} \{z_j\ne 0\}\right)\subset \C^m.
\]

Our construction of a cell decomposition of $X_{\Sigma, \Lambda}$ is a slight generalization of the Morse-theoretic argument introduced in~\cite{Kh86} and applied to quasitoric manifolds in~\cite{davis1991convex}. Since we do not assume that $\Sigma$ is a normal fan for a certain polytope, we cannot use the generic linear functions as in~\cite{Kh86}. Instead, let us choose a vector $v\in \R^n$ in a general position with respect to $\Sigma$, i.e. a vector $v$ which belongs to the interior of a full-dimensional cone of $\Sigma$.

Let $\tau_1,\ldots, \tau_s$ be cones of dimension $n$ in $\Sigma$. For a maximal cone $\tau$, we will say that a face $\sigma$ of $\tau$ is {\emph{incoming}} with respect to the vector $v$ if the intersection $\tau \cap (\sigma+v)$ is unbounded. Let us further define the index $\ind(\tau)$ of a maximal cone $\tau$ to be the number of incoming rays of $\tau$.

To each maximal cone $\tau$ we associate a disjoint union of open cells of $\mathcal Z_K$ via the formula:
\[
\widetilde U_\tau = \bigsqcup_{\sigma} H_\sigma,
\]
where the union is taken over all incoming faces $\sigma$ of $\tau$. Since each cone $\sigma$ is incoming for a unique cone  $\tau$ of maximal dimension, we get a cell decomposition:
\[
\mathcal Z_K = \bigsqcup_{i=1}^s \widetilde U_{\tau_i}.
\]
It is easy to see that the cells $U_\tau$ are invariant under the action of $H \simeq (S^1)^{m-n}$ and that
\[
\widetilde U_\tau \simeq (D^2)^{\ind(\tau)} \times (S^1)^{m-n}.
\]
Moreover, the action of $H$ is free and transitive on the second factor in $(D^2)^{\ind(\tau)} \times (S^1)^{m-n}$, hence we get:
\[
X_{\Sigma,\Lambda}=\bigsqcup_{i=1}^s \widetilde U_{\tau_i}/H,
\]
where $\widetilde U_{\tau_i}/H \simeq (D^2)^{\ind(\tau_i)}$

\begin{theorem}\label{thm:celldecom}
Let $X_{\Sigma,\Lambda}$ be a generalized quasitoric manifold. Then $X_{\Sigma,\Lambda}$ has a cellular decomposition with only even-dimensional cells. The cells in this decomposition are in a bijection with maximal cones $\tau$ in $\Sigma$. The dimension of the cell corresponding to a cone $\tau$ is $2\,\ind(\tau)$.
\end{theorem}
\begin{corollary}
The Euler characteristic of the manifold $X_{\Sigma,\Lambda}$ is equal to the number of maximal cones in $\Sigma$.
\end{corollary}

\subsection{Relations between characteristic submanifolds.}
In this subsection we will deduce two types of relations between classes of codimension-2 characteristic submanifolds in the intersection ring of a manifold $X_{\Sigma,\Lambda}$. In the following proposition we show that the Stanley-Raisner relations hold in $H^*(X_{\Sigma,\Lambda})$.

\begin{proposition}\label{prop:SR}
For codimension-2 characteristic submanifolds $D_{i_1},\ldots,D_{i_n}$, in the cohomology ring of a generalized quasitoric manifold $X_{\Sigma,\Lambda}$ one has:
\[
[D_{i_1}]\cdots[D_{i_n}] = 
\begin{cases}
\s(I)[X_{\Sigma,\Lambda}]^\ast, \quad\text{ if } \rho_{i_1}\ldots,\rho_{i_n} \text{ form a cone in } \Sigma\\
0, \qquad \qquad\text{otherwise} 
\end{cases}
\]
\end{proposition}
\begin{proof}
Indeed, in the cohomology ring of the generalized quasitoric manifold $X_{\Sigma,\Lambda}$ we have: $[D_{i_1}]\cdots[D_{i_n}]=(-1)^v[X_{\Sigma,\Lambda}]^\ast$, where we denote by $(-1)^v$ the sign of the fixed point $v=D_{i_1}\cap\cdots\cap D_{i_n}\in X_{\Sigma,\Lambda}$, which compares two orientations on $\mathcal T_{v}X_{\Sigma,\Lambda}$: the one induced by coorientations of characteristic submanifolds $D_i$ and the one induced by the representation of $T^n:=T^m/H$ in the tangent space $\mathcal T_{v}X_{\Sigma,\Lambda}\cong\C^n$. 

On the other hand, the weights of the tangential representation of the compact torus $T^n$ at the fixed point $v$ form a lattice basis dual to $(\Lambda(\rho_{i_1}),\ldots,\Lambda(\rho_{i_n}))$. Therefore, the sign $(-1)^{v}=\det(\Lambda(\rho_{i_1}),\ldots, \Lambda(\rho_{i_n}))=\s(I)$, which finishes the proof.
\end{proof}

To obtain the linear relations we need to analyze further the construction of generalized quasitoric manifolds. There are natural $(S^1)^m$-equivariant line bundles $L_1,\ldots, L_m$ on $\mathcal Z_K$. For each integer vector ${\bf k}=(k_1,\ldots, k_m) \in \Z^m$, the tensor product
\[
L_{{\bf k}}= L_1^{k_1}\otimes \ldots \otimes L_m^{k_m}
\]
descends to a complex line bundle $\widetilde L_{{\bf k}}$ on $X_{\Sigma, \Lambda}$. Moreover, if ${\bf k}\in \Z^m$ is such that the corresponding character acts trivially on $H_\Lambda\subset (S^1)^m$, the descendant bundle $\widetilde L_{{\bf k}}$ is topologically trivial. 

It is easy to see that there is a smooth section of $\widetilde L_{{\bf k}}$ with the degenerate locus given by $\sum_{i=1}^m k_i[D_i]$. By exactness of the sequence
\[
0\to M \xrightarrow{\Lambda^*} \Z^m \to M_{H_\Lambda} \to 0,
\]
the characters $\bf{k}$ acting trivially on $H_\Lambda$ are identified with the character lattice $M$ of $T$ with $k_i = \chi(v_i)$ for $\chi\in M$ and $v_i=\Lambda(\rho_i)$. Thus we obtain the following proposition.

\begin{proposition}\label{prop:lin}
For any character $\chi\in M$, the following linear relation in $H^2(X_{\Sigma,\Lambda})$ holds:
$$
 \sum_{i=1}^m \chi(v_i)[D_i] = 0,
$$
where $v_i:=\Lambda(\rho_i)$, for $1\leq i\leq m$.
\end{proposition}
\begin{proof}
Indeed, the descendant complex line bundle $\widetilde L_{\chi(v_1),\ldots,\chi(v_m)}$ is trivial and hence its first Chern class is equal to zero:
\[
c_1(\widetilde L_{\chi(v_1),\ldots,\chi(v_m)}) =  \sum_{i=1}^m \chi(v_i)[D_i] = 0. \qedhere
\]
\end{proof}

\subsection{Topological version of the BKK Theorem} 
Let us start with an important observation: to describe a cohomology ring of a generalized quasitoric manifold it is enough to compute the {\emph{self-intersection polynomial}}:
\[
h_1[D_1] + \ldots + h_m[D_m] \mapsto \left\langle (h_1[D_1] + \ldots + h_m[D_m])^m, [X_{\Sigma,\Lambda}] \right\rangle 
\]
on the space of all linear combinations of classes of codimemsion~$2$ characteristic submanifolds. This is the subject of the following theorem.

\begin{theorem}\label{thm:BKK}
Let $X_{\Sigma,\Lambda}$ be a generalized quasitoric manifold with codimension~$2$ characteristic submanifolds $D_1,\ldots, D_m$. Then the following identity holds
$$
\left\langle (h_1[D_1] + \ldots + h_m[D_m])^m, [X_{\Sigma,\Lambda}] \right\rangle  = n!Vol(f_{h}),
$$
where $f_h \in \Pm_{\Sigma,\Lambda}$ is a generalized virtual polytope associated to the simplicial complex $K_\Sigma$, characteristic map $\Lambda$ and the set of parameters $h=(h_1,\ldots,h_m)$.
\end{theorem}
\begin{proof}
Let us identify the space of all linear combinations $h_1[D_1] + \ldots + h_m[D_m]$ with the space of generalized virtual polytopes $\Pm_{\Sigma,\Lambda}$. Under this identification, both self-intersection and volume functions are homogeneous polynomials of degree $n$ on $\Pm_{\Sigma,\Lambda}$. Let us denote them by $S\colon\Pm_{\Sigma,\Lambda}\to \R$ and $Vol\colon\Pm_{\Sigma,\Lambda}\to \R$, respectively.

To show the equality $S(h)=n! Vol(h)$ holds, it is enough to prove the equality of all (mixed) partial derivatives of $S$ and $Vol$ of degree $n$: 
\[
\partial_{i_1}^{k_1}\ldots \partial_{i_s}^{k_s}S(h) = n! \cdot \partial_{i_1}^{k_1}\ldots \partial_{i_s}^{k_s} Vol(h),
\]
where $\partial_{i_j} = \partial/\partial h_{i_j}$ and $\sum_{j=1}^s k_{i_j}=n$.

Let us call the number $\sum_{i=1}^s (k_i-1)$ the \emph{multiplicity of the monomial} $\partial_{i_1}^{k_1}\dots\partial_{i_s}^{k_s}$.
In particular, a monomial has multiplicity $0$ if and only if it is square free. We will prove the equality of mixed partial derivatives by induction on the multiplicity of a differential monomial. 

For square free monomials, the equality follows from the first part of Corollary~\ref{volpolynomialderivativescoro} and Proposition~\ref{prop:SR}. Indeed, by Corollary~\ref{volpolynomialderivativescoro} in the case when $r = n$ and the vertices $v_{i_1}, \ldots, v_{i_n}$ form a simplex in $\Delta$ dual to the vertex $A\in \Delta^{\perp}$, we have: 
  \[
    \partial_{i_1}\ldots\partial_{i_n} \Vol(h) = \begin{cases}
    \s(i_1,\ldots, i_n), \quad \text{ if } \rho_{i_1}, \ldots, \rho_{i_n} \text{ span a cone in } \Sigma;\\
    0, \quad \quad \quad \quad \quad \quad \quad \text{otherwise.}
    \end{cases}
  \]
On the other hand $\partial_{i_1}\ldots\partial_{i_n} S(h)$ is equal to the coefficient in front of $t_{i_1}\ldots,t_{i_n}$ in the polynomial $S(h+(t_1\ldots,t_m))$. We get:
\[
S(h+(t_1\ldots,t_n)) = \left\langle ((h_1+t_1)[D_1] + \ldots + (h_m+t_m)[D_m])^m, [X_{\Sigma,\Lambda}] \right\rangle = t_{i_1}\ldots t_{i_n} \cdot n!\cdot \left\langle D_{i_1}\ldots D_{i_n}, [X_{\Sigma,\Lambda}] \right\rangle+\ldots
\]
Hence by Proposition~\ref{prop:SR} we get:
\[
\partial_{i_1}\ldots\partial_{i_n} S(h) = \begin{cases}
    n!\cdot\s(i_1,\ldots, i_n), \quad \text{ if } \rho_{i_1}, \ldots, \rho_{i_n} \text{ span a cone in } \Sigma;\\
    0, \quad \quad \quad \quad \quad \quad \quad \quad \,\,\,\text{otherwise.}
    \end{cases}
\]

Now, let us assume that the equality of mixed partial derivatives holds for all differential monomials of multiplicity $r-1$. Let $\partial_{i_1}^{k_1}\ldots \partial_{i_s}^{k_s}$ be a differential monomial of multiplicity $r$ with $k_1\geq 1$. We can assume that $\rho_{i_1}, \ldots, \rho_{i_s}$ span a cone in $\Sigma$, since otherwise
\[
\partial_{i_1}^{k_1}\ldots \partial_{i_s}^{k_s}S(h) = n! \cdot \partial_{i_1}^{k_1}\ldots \partial_{i_s}^{k_s} Vol(h) = 0.
\]
In that case, there exists a character $\chi\in M$ such that 
\[
\langle\chi, \Lambda(\rho_{i_1})\rangle =1,\,\langle\chi, \Lambda(\rho_{i_2})\rangle =0,\, \ldots,\, \langle\chi, \Lambda(\rho_{i_s})\rangle =0.
\]
Therefore, since the volume is invariant under the translation of a generalized virtual polytope, we get
\[
\partial_{i_1}^{k_1}\ldots \partial_{i_s}^{k_s} Vol(h) = -\sum_{l\ne i_j}\langle\chi, \Lambda(\rho_{l})\rangle\partial_l\partial_{i_1}^{k_1-1}\ldots \partial_{i_s}^{k_s} Vol(h) 
\]
and similarly by Proposition~\ref{prop:lin}:
\[
\partial_{i_1}^{k_1}\ldots \partial_{i_s}^{k_s} S(h) = -\sum_{l\ne i_j}\langle\chi, \Lambda(\rho_{l})\rangle\partial_l\partial_{i_1}^{k_1-1}\ldots \partial_{i_s}^{k_s} S(h). 
\]
Moreover, the differential monomials on the right hand side of the expressions above have multiplicities less than $r$, so the equality 
\[
\partial_{i_1}^{k_1}\ldots \partial_{i_s}^{k_s}S(h) = n! \cdot \partial_{i_1}^{k_1}\ldots \partial_{i_s}^{k_s} Vol(h)
\]
follows from the induction hypothesis.
\end{proof}

We will finish this subsection by providing a different interpretation of Theorem~\ref{thm:BKK}. Let us first recall the classical interpretation of the BKK Theorem for toric varieties. The {\emph{Newton polyhedron}} $\Delta(f)\subset \R^n$ of a Laurent polynomial $f=\sum a_i x^{k_i}$ is the convex hull of the vectors $k_i$ with $a_i\ne 0$. For a fixed polytope $\Delta$, let $E_\Delta$ be a finite-dimensional vector space of Laurent polynomials $f$ such that $\Delta(f)\subset\Delta$. 

%The BKK Theorem computes the number of solutions in $(\C^*)^n$ of a system $f_1 = \ldots = f_n = 0$ of generic Laurent polynomials %with fixed Newton polytopes $\Delta_1,\ldots, \Delta_n$. 

\begin{theorem}[BKK Theorem]
Let $f_1, \ldots, f_n$ be generic Laurent polynomials with $f_i\in E_{\Delta_i}$, for $1\leq i\leq n$. Then all the solutions of the system $f_1 = \ldots = f_n = 0$ in $(\C^*)^n$ are non-degenerate and the number of solutions is equal to
$$
n! Vol(\Delta_1,\ldots,\Delta_n),
$$
where $Vol$ is the mixed volume of virtual polytopes function.
\end{theorem}

One can reformulate Theorem~\ref{thm:BKK} in a similar way. Let $\Delta_1,\ldots, \Delta_n$ be generalized virtual polytopes in $\Pm_{\Sigma,\Lambda}$ associated to a generalized quasitoric manifold $X_{\Sigma,\Lambda}$. Let $L_{\Delta_i}$ be a line bundle associated to the generalized virtual polytope $\Delta_i$ and let $E_{\Delta_i}=\Gamma(X_{\Sigma,\Lambda},L_{\Delta_i})$ be the space of smooth sections of $L_{\Delta_i}$. Then Theorem~\ref{thm:BKK} can be reformulated in the following way.

\begin{theorem}
Let $s_1, \ldots, s_n$ be generic Laurent polynomials with $s_i\in E_{\Delta_i}$, for $1\leq i\leq n$. Then all the solutions of the system $s_1 = \ldots = s_n = 0$ in $X_{\Sigma,\Lambda}$ are non-degenerate and the number of solutions counted with signs is equal to
$$
n! Vol(\Delta_1,\ldots,\Delta_n),
$$
where $Vol$ is the mixed volume of generalized virtual polytopes function.
\end{theorem}

\begin{remark}
Note that in the algebraic case, the multiplicity of each non-degenerate root is equal to $1$; however, in the case of smooth sections $s_i\in \Gamma(X_{\Sigma,\Lambda},L_{\Delta_i})$, the multiplicity of a non-degenerate root might be equal to $-1$. Nevertheless, the number of solutions counted with signs can still be computed in terms of a mixed volume.
\end{remark}

%%%%%%%%%%%%%%%%%%%%%%%%%%%%%%%%%%%%%%%%%%%%%%%%%%%%%%%%%%%%%%%%%%%%%%%%%%%%%%%%%%%%%%%%%%%%%%%%%%%%%%%%%%%%%%%%%%%%%%%%%%

\subsection{Pukhlikov-Khovanskii type description} In this subsection we use the approach introduced by {Pukh\-li\-kov} and the first author for the computation of cohomology rings. The key ingredient of such a description is an exact computation of Macaulay inverse systems for graded algebras with Poincar\'e duality generated in degree~$1$.

We will call a graded, commutative algebra $A = \bigoplus_{i=0}^n A_i$ over a field $\K$ of characteristic $0$ a \emph{Poincar\'e duality algebra} if 
\begin{itemize}
    \item $A_0 \simeq A_n \simeq \K$;
    \item the bilinear map $A_i \times A_{n-i} \to A_n$ is non-degenerate for any $i = 0, \ldots, n$ (Poincar\'{e} duality).
\end{itemize}  
The main example of a Poincar\'e duality algebra arises as follows. Let $X$ be a smooth closed orientable manifold of dimension $2n$. Then the algebra of even-degree cohomology classes $A = \bigoplus_{i=0}^n H^{2i}(X)$ is a Poincar\'e duality algebra. In particular, since for a generalized quasitoric manifold $X_{\Sigma,\Lambda}$ one has $H^{2i+1}(X_{\Sigma,\Lambda})=0$ for all $i\geq 0$, its cohomology ring $H^*(X_{\Sigma,\Lambda})$ is also a Poincar\'e duality algebra. The next theorem yields a description of Poincar\'e duality algebras. 

\begin{theorem}\label{Pkh}
 Let $A$ be a Poincar\'e duality algebra generated (as an algebra) by the elements from $A_1=\K\langle v_1,\ldots, v_r\rangle$ (i.e. by elements of degree one). Then 
  \[
    A\simeq \K[t_1, \ldots, t_r]/\{p(t_1, \ldots, t_r) \in \K[t_1, \ldots, t_r] \colon p(\tfrac{\partial}{\partial x_1}, \ldots, \tfrac{\partial}{\partial x_r}) f(x_1, \ldots, x_r) = 0 \},
  \]
  where we identify $A_1$ with $\K^r$ via a basis $v_1, \ldots, v_r$ and $f\colon A_1 \simeq \K^r \to \K$ is a polynomial given by the formula: 
  \[
    f(x_1,\ldots, x_r) = (x_1v_1 + \ldots + x_rv_r)^n \in A_n \simeq k.
  \]
\end{theorem}

Theorem~\ref{Pkh} was used in \cite{PK92} to give a description of the cohomology ring of a smooth projective toric variety. Later, it was used in \cite{KavehVolume} to provide a description of cohomology rings of full flag varieties $G/B$. A more general version of Theorem~\ref{Pkh} has been obtained recently in~\cite{KhoM} and used in~\cite{hof2020,KLM21} to give a description of cohomology rings of toric and quasitoric bundles.

Theorem~\ref{Pkh} accepts a coordinate free reformulation. Indeed, the ring $\K[t_1, \ldots, t_r]$ in Theorem~\ref{Pkh} can be identified with the ring of differential operators with constant coefficients $\Diff(A_1)$ on $A_1$. Hence the description of the algebra $A$ becomes 
\[
A\simeq \Diff(A_1)/\Ann(f),
\]
where $\Ann(f) = \{D\in \Diff(A_1) \,|\, D\cdot f=0$ is the annihilator ideal of $f$.

\begin{theorem}\label{PKHquasitoric:thm}
Let $X_{\Sigma,\Lambda}$ be a generalized quasitoric manifold and let $\Pm_{\Sigma,\Lambda}$ be the space of generalized virtual polytopes associated to it. Then the cohomology ring $H^*(X_{\Sigma,\Lambda})$ can be computed as
\[
H^*(X_{\Sigma,\Lambda}) = \Diff(\Pm_{\Sigma,\Lambda})/ \Ann(\Vol),
\]
where $\Diff(\Pm_{\Sigma,\Lambda})$ is the ring of differential operators with constant coefficients on $\Pm_{\Sigma,\Lambda}$ and $\Ann(\Vol)$ is the annihilator ideal of the volume polynomial.
\end{theorem}
\begin{proof}
By Theorem~\ref{thm:celldecom}, the cohomology ring $H^*(X_{\Sigma,\Lambda})$ is generated by the classes of codimension-2 characteristic submanifolds in $X_{\Sigma,\Lambda}$. Hence there exists a surjection $\Diff(\Pm_{\Sigma,\Lambda})\to H^*(X_{\Sigma,\Lambda})$ with a kernel given, by Theorem~\ref{Pkh}, as the annihilator ideal of the self-intersection polynomial $S(h)$ of classes of codimension-2 characteristic submanifolds. However, by Theorem~\ref{thm:BKK}, $S(h)=n!\Vol(h)$ and hence:
\[
H^*(X_{\Sigma,\Lambda}) = \Diff(\Pm_{\Sigma,\Lambda})/ \Ann(S) = \Diff(\Pm_{\Sigma,\Lambda})/ \Ann(\Vol).
\]
\end{proof}

%\bibliographystyle{alpha}
%\bibliography{tbc}

%\Addresses

\end{document}